\documentclass[a4paper,11pt]{amsart}

\usepackage{comment}
\usepackage{amsthm}
\usepackage{amssymb}
\usepackage{amsmath}
\usepackage{accents}
\usepackage[hidelinks]{hyperref}
\usepackage{xcolor}
\usepackage{enumerate}
\usepackage{listings}
\usepackage{blkarray}
\usepackage{braket}
\usepackage{lmodern}
\usepackage[english]{babel}
\usepackage[capitalize]{cleveref}
\usepackage{todonotes,mathtools}
\usepackage[numbers,sort]{natbib}
\usepackage{enumitem,fullpage}

\usepackage[font=small,labelfont=bf]{caption}
\usepackage[font=small,labelfont=normalfont,labelformat=simple]{subcaption}
\graphicspath{{figs/}}

\newtheorem{theorem}{Theorem}[section]
\newtheorem{lemma}[theorem]{Lemma}
\newtheorem{claim}[theorem]{Claim}

\newtheorem{observation}[theorem]{Observation}

\newtheorem{conjecture}[theorem]{Conjecture}

\theoremstyle{definition}
\newtheorem{definition}[theorem]{Definition}

\DeclareMathOperator{\pr}{Pr}

\crefname{claim}{Claim}{Claims}

\newlist{lemenum}{enumerate}{1}
\setlist[lemenum]{label=(\alph*), ref=\thelemma(\alph*)}
\crefalias{lemenumi}{lemma}

\renewcommand{\v}{{\mathsf{v}}}
\newcommand{\e}{{\mathsf{e}}}

\author{Micha Christoph}
\address{Department of Computer Science, Institute of Theoretical Computer Science, ETH Z\"{u}rich, Switzerland}
\email{micha.christoph@inf.ethz.ch}
\author{Anders Martinsson}
\address{Department of Computer Science, Institute of Theoretical Computer Science, ETH Z\"{u}rich, Switzerland}
\email{anders.martinsson@inf.ethz.ch}
\author{Raphael Steiner}
\address{Department of Computer Science, Institute of Theoretical Computer Science, ETH Z\"{u}rich, Switzerland}
\email{raphaelmario.steiner@inf.ethz.ch}
\author{Yuval Wigderson}
\address{Institute for Theoretical Studies, ETH Z\"urich, Switzerland}
\email{yuval.wigderson@eth-its.ethz.ch}

\thanks{Research of M.C. and R.S. funded by SNSF Ambizione grant No. 216071. \\ Research of Y.W. supported by Dr.\ Max R\"ossler, the Walter Haefner Foundation, and the ETH Z\"urich Foundation.}

\date{\today}

\title{Resolution of the Kohayakawa--Kreuter conjecture}

\begin{document}

\begin{abstract}
A graph $G$ is said to be \emph{Ramsey} for a tuple of graphs $(H_1,\dots,H_r)$ if every $r$-coloring of the edges of $G$ contains a monochromatic copy of $H_i$ in color $i$, for some $i$. A fundamental question at the intersection of Ramsey theory and the theory of random graphs is to determine the threshold at which the binomial random graph $G_{n,p}$ becomes a.a.s.\ Ramsey for a fixed tuple $(H_1,\dots,H_r)$, and a famous conjecture of Kohayakawa and Kreuter predicts this threshold. Earlier work of Mousset--Nenadov--Samotij, Bowtell--Hancock--Hyde, and Kuperwasser--Samotij--Wigderson has reduced this probabilistic problem to a \emph{deterministic} graph decomposition conjecture. In this paper, we resolve this deterministic problem, thus proving the Kohayakawa--Kreuter conjecture. Along the way, we prove a number of novel graph decomposition results which may be of independent interest.

MSC classifications: 05C70, 05D10 (primary), 05C80 (secondary).
\end{abstract}
\maketitle

\section{Introduction}
The topic of this paper is that of graph decompositions, which we view from two different perspectives. From a graph-theoretic perspective, a graph decomposition result is a statement of the form ``any graph $G$ satisfying certain properties can be edge-partitioned into subgraphs with a certain structure''. A famous result of this type is Nash-Williams' theorem \cite{MR0161333}, which gives a necessary and sufficient sparsity condition for when $G$ can be decomposed into the edge-union of $k$ forests.

Additionally, we also study graph decompositions from a Ramsey-theoretic perspective, in which a typical statement is of the form ``if a graph $G$ is sufficiently complex, then it cannot be edge-partitioned into subgraphs avoiding a certain structure''. The foundational result of this type is Ramsey's theorem \cite{MR1576401}, which can be phrased as saying that any sufficiently large complete graph cannot be edge-partitioned into two subgraphs of bounded clique number. Thus, in some sense, the graph-theoretic perspective on decomposition asks when a decomposition into certain structures is \emph{possible}, and the Ramsey-theoretic perspective asks when certain structures are \emph{unavoidable} in \emph{every} decomposition.

The main results of this paper interact with both perspectives. 
On the one hand, we resolve a conjecture of Kohayakawa and Kreuter \cite{MR1609513} on Ramsey properties of random graphs. On the other hand, the proof uses the first perspective. 
% However, the proof uses an edge-partition based on certain densities of subgraphs, which lies in the realm of the first perspective. 
For example, one of our main theorems, \cref{thm:43} below, states that if a graph satisfies a certain sparsity condition (closely related to the one appearing in Nash-Williams' theorem), then it can be edge-partitioned into a forest and a subgraph satisfying a different, stronger sparsity condition. This statement and its proof are closely related to many known graph decomposition theorems, such as the main results of \cite{2310.00931,MR4125897,MR4670576,MR0180501,MR0161333,MR3759911,MR3383251}. As mentioned above, the reason we care about the precise statement comes from the Ramsey-theoretic study of random graphs, a topic we now introduce.

\subsection{Ramsey properties of random graphs}
Given an $r$-tuple $(H_1,\dots,H_r)$ of graphs, a graph $G$ is said to be \emph{Ramsey} for $(H_1,\dots,H_r)$ if every $r$-coloring of the edges of $G$ contains a monochromatic copy of $H_i$ in color $i$, for some $1 \leq i \leq r$. Equivalently, in the language above, $G$ is Ramsey for $(H_1,\dots,H_r)$ if there is no decomposition of $G$ into the edge-union of $r$ graphs, the $i$th of which is $H_i$-free\footnote{We say that a graph is $H$-free if it does not contain a copy of $H$ as a (not necessarily induced) subgraph.}. In general, the fundamental question of graph Ramsey theory is to understand which graphs $G$ are Ramsey for a given $r$-tuple; for more information, see e.g.\ the survey \cite{MR3497267}. An important special case of this question, which came to prominence in the late 1980s thanks to pioneering work of Frankl--R\"odl \cite{MR0932121} and \L{}uczak--Ruci\'nski--Voigt \cite{MR1182457}, is the question of when a \emph{random} graph is Ramsey for a given $r$-tuple with high probability. More precisely, if $G_{n,p}$ denotes the binomial random graph with edge density $p$, then the question is for which values of $p=p(n)$ one has that $G_{n,p}$ is Ramsey for $(H_1,\dots,H_r)$ a.a.s.\footnote{We say that an event happens \emph{asymptotically almost surely (a.a.s.)}\ if its probability tends to $1$ as $n \to \infty$, where the implicit parameter $n$ will always be clear from context.} In the symmetric case $H_1 = \dotsb = H_r$, this question was completely resolved in seminal work of R\"odl and Ruci\'nski \cite{MR1276825,MR1249720,MR1262978}. For a graph $J$, let us denote by $\v(J),\e(J)$ its number of vertices and edges, respectively, and let us define the \emph{maximal 2-density} of a graph $H$ with\footnote{By convention, one also defines $m_2(H)=0$ if $\e(H)=0$ and $m_2(H)=\frac 12$ if $\e(H)=1$.} $\e(H) \geq 2$ to be
\[
m_2(H) \coloneqq \max \left\{ \frac{\e(J) - 1}{\v(J) - 2} : J \subseteq H, \v(J) \geq 3\right\}.
\]
With this notation, the random Ramsey theorem of R\"odl and Ruci\'nski \cite{MR1276825,MR1249720} is as follows.
\begin{theorem}[R\"odl--Ruci\'nski]\label{thm:RR}
    Let $H$ be a graph which is not a forest and let $r\geq 2$ be an integer. There exist constants $C>c>0$ such that
    \[
    \lim_{n \to \infty} \pr(G_{n,p} \text{ is Ramsey for }(\underbrace{H,\dots,H}_{r\text{ times}})) = 
    \begin{cases}
        1 & \text{if }p \geq Cn^{-1/m_2(H)},\\
        0&\text{if }p \leq cn^{-1/m_2(H)}.
    \end{cases}
    \]
\end{theorem}
Any theorem of this form really consists of two statements, which are usually called the $1$- and $0$-statements; the latter says that the asymptotic probability of an event is $0$ in a certain regime, and the former says that the asymptotic probability is $1$ in a different regime.

There is a simple heuristic explanation for why the threshold for the Ramsey property is controlled by the quantity $m_2(H)$. To explain it, let us suppose for simplicity that $H$ is \emph{strictly 2-balanced}, meaning that $m_2(J) < m_2(H)$ for any proper subgraph $J \subsetneq H$. Then one can easily verify that at the regime $p \asymp n^{-1/m_2(H)}$, an average edge of $G_{n,p}$ lies in a constant number of copies of $H$. Thus, if $p \leq cn^{-1/m_2(H)}$ where $c \ll 1$, then a typical edge lies in no copy of $H$, and one expects the copies of $H$ in $G_{n,p}$ to be ``well spread-out''. Thus, it is reasonable to expect that one can $r$-color the edges without creating monochromatic copies of $H$. On the other hand, if $p \geq Cn^{-1/m_2(H)}$ where $C \gg 1$, then a typical edge lies in a large (constant) number of copies of $H$. In this regime, we expect a lot of interaction between different $H$-copies, and it should be difficult to avoid creating monochromatic copies.

The heuristic above suggests that the ``reason'' why $G_{n,p}$ should be Ramsey for $(H,\dots,H)$ is because of a ``global'' interaction between the copies of $H$. However, there is also a potential ``local'' reason: for any fixed graph $G$ that is Ramsey for $(H,\dots,H)$, if $G$ is a subgraph of $G_{n,p}$, then certainly $G_{n,p}$ is Ramsey for $(H,\dots,H)$. Therefore, in order to prove the $0$-statement in \cref{thm:RR}, one necessarily has to prove that when $p \leq cn^{-1/m_2(H)}$, then $G_{n,p}$ a.a.s.\ does not contain any fixed $G$ which is Ramsey for $(H,\dots,H)$. It is well-known (e.g.\ \cite[Theorem 3.4]{MR1782847}) that the threshold for appearance of $G$ in $G_{n,p}$ is determined by the \emph{maximal density} of $G$, defined as
\[
m(G) \coloneqq \max \left\{ \frac{\e(J)}{\v(J)} : J \subseteq G, \v(J) \geq 1\right\}.
\]
Thus, a necessary condition for the $0$-statement in \cref{thm:RR} is the following lemma, which R\"odl and Ruci\'nski \cite{MR1249720} termed the \emph{deterministic lemma}.
\begin{lemma}[R\"odl--Ruci\'nski]
    Let $H$ be a graph which is not a forest, and let $r \geq 2$ be an integer. If $G$ is Ramsey for the $r$-tuple $(H,\dots,H)$, then $m(G) > m_2(H)$.
\end{lemma}
Equivalently, the deterministic lemma can be phrased as a decomposition result: if $m(G) \leq m_2(H)$, then $G$ can be edge-partitioned into $r$ graphs, each of which is $H$-free. As discussed above, the deterministic lemma is certainly a necessary condition for the $0$-statement in \cref{thm:RR} to hold, but in fact, R\"odl and Ruci\'nski \cite{MR1249720} proved a so-called \emph{probabilistic lemma}, which states that this simple necessary condition is also sufficient. Informally, the probabilistic lemma guarantees that the ``local'' and ``global'' reasons discussed above are the only two reasons why $G_{n,p}$ is Ramsey for $(H,\dots,H)$, and the deterministic lemma rules out the ``local'' reason.

\cref{thm:RR} provides a very satisfactory answer to the question ``when is $G_{n,p}$ Ramsey for $(H_1,\dots,H_r)$?'' in the case that $H_1 = \dotsb = H_r$, but says nothing about the general case. However, nearly thirty years ago, Kohayakawa and Kreuter \cite{MR1609513} formulated a conjecture for the threshold for an arbitrary $r$-tuple of graphs. Given two graphs $H_1,H_2$ with $m_2(H_1) \geq m_2(H_2)$, they defined the \emph{mixed 2-density} to be
\[
m_2(H_1,H_2) \coloneqq \max \left \{ \frac{\e(J)}{\v(J) - 2 + 1/m_2(H_2)} : J \subseteq H_1, \v(J) \geq 2\right\}.
\]
It is well-known and easy to verify (see e.g.\ \cite[Lemma 3.4]{kuper} or \cite[Proposition 3.1]{bowtell}) that $m_2(H_1)\geq m_2(H_1,H_2) \geq m_2(H_2)$, and that both inequalities are strict if one is.
\begin{conjecture}[Kohayakawa--Kreuter]\label{conj:KK}
    Let $H_1,\dots,H_r$ be graphs, and suppose that $m_2(H_1) \geq \dotsb \geq m_2(H_r)$ and $m_2(H_2)>1$. There exist constants $C>c>0$ such that
    \[
    \lim_{n \to \infty} \pr(G_{n,p} \text{ is Ramsey for }(H_1,\dots,H_r)) = 
    \begin{cases}
        1 & \text{if }p \geq Cn^{-1/m_2(H_1,H_2)},\\
        0&\text{if }p \leq cn^{-1/m_2(H_1,H_2)}.
    \end{cases}
    \] 
\end{conjecture}
We remark that the assumption $m_2(H_2)>1$ is equivalent to saying that $H_2$ is not a forest; this condition was added by Kohayakawa, Schacht, and Sp\"ohel \cite{MR3143588} to rule out sporadic counterexamples (in analogy to the statement of \cref{thm:RR}). Just as in the case of \cref{thm:RR}, there is a simple heuristic explanation for why the function $m_2(H_1,H_2)$ controls the threshold for the asymmetric Ramsey property of $G_{n,p}$. Roughly speaking, one can verify that at the density $p \asymp n^{-1/m_2(H_1,H_2)}$, the number of edges appearing in the union of the $H_1$-copies is of the same order as the number of copies of $H_2$. As the edges not participating in copies of $H_1$ are irrelevant (they can be colored in color $1$ with no adverse consequences), we find that an average ``relevant'' edge intuitively lies in a constant number of copies of $H_2$, and the heuristic above suggests that this is the threshold at which it becomes hard to avoid monochromatic copies of $H_2$. We remark that this heuristic, as well as the statement of \cref{conj:KK}, does not involve $H_3,\dots,H_r$ at all; thus, the intuition is that being Ramsey for a tuple is essentially as hard as being Ramsey for the two densest graphs in the tuple.

\cref{conj:KK} has received a great deal of attention over the past three decades \cite{MR1609513, MR2531778, MR3143588, MR3725732, MR3899160, kuper, bowtell, 2305.19964}. For many years, most papers on the topic aimed to prove the Kohayakawa--Kreuter conjecture for certain special families of $H_1,\dots,H_r$; for example, Kohayakawa and Kreuter \cite{MR1609513} proved it when every $H_i$ is a cycle, Marciniszyn, Skokan, Sp\"ohel, and Steger \cite{MR2531778} proved it when every $H_i$ is a clique, Liebenau, Mattos, Mendon\c ca, and Skokan \cite{MR4597167} proved it when $H_1$ is a clique and $H_2$ is a cycle, Hyde \cite{MR4565396} proved it for most pairs of regular graphs, and Kuperwasser and Samotij \cite{2305.19964} proved it when $m_2(H_1)=m_2(H_2)$. More recent works have proved results in greater generality. Notably, Mousset, Nenadov, and Samotij \cite{MR4173138} established the $1$-statement of \cref{conj:KK} for all $(H_1,\dots,H_r)$. Subsequently, Bowtell--Hancock--Hyde \cite{bowtell} and Kuperwasser--Samotij--Wigderson \cite{kuper} independently proved a generalization of the R\"odl--Ruci\'nski probabilistic lemma in the asymmetric case.
\begin{theorem}[Bowtell--Hancock--Hyde, Kuperwasser--Samotij--Wigderson]\label{thm:prob lem}
    Suppose that for every pair of graphs $(H_1,H_2)$ with $m_2(H_1) > m_2(H_2) > 1$, the following holds: if $G$ is Ramsey for $(H_1,H_2)$, then $m(G) > m_2(H_1,H_2)$. Then \cref{conj:KK} is true\footnote{Strictly speaking, the result of \cite{kuper,bowtell} only implies that the $0$-statement of \cref{conj:KK} is true, but the $1$-statement is already known to be true unconditionally, thanks to the result of Mousset--Nenadov--Samotij \cite{MR4173138} mentioned above.}.
\end{theorem}
\cref{thm:prob lem} is a generalization of the R\"odl--Ruci\'nski probabilistic lemma discussed above. Indeed, the condition in \cref{thm:prob lem} is clearly necessary for the $0$-statement of \cref{conj:KK} to hold, as if there were some $G$ with $m(G) \leq m_2(H_1,H_2)$ which is Ramsey for $(H_1,H_2)$, then that $G$ appears with positive probability in $G_{n,cn^{-1/m_2(H_1,H_2)}}$ for any constant $c>0$, and hence the $0$-statement of \cref{conj:KK} would be false. \cref{thm:prob lem} then says that this necessary condition is also sufficient. Thus, the validity of \cref{conj:KK} is reduced to a deterministic graph decomposition question. In both \cite{bowtell,kuper}, this deterministic condition was verified for most pairs $(H_1,H_2)$, but its verification for all pairs remained open.

\subsection{Our results}
As discussed above, the earlier works \cite{MR4173138,bowtell,kuper} reduced  \cref{conj:KK} to verifying a certain deterministic condition.
Our main result confirms that this condition always holds, thus completing the proof of \cref{conj:KK}.
\begin{theorem}\label{thm:det lem}
    Let $H_1,H_2$ be graphs with $m_2(H_1) > m_2(H_2)>1$. If a graph $G$ is Ramsey for $(H_1,H_2)$, then $m(G) > m_2(H_1,H_2)$.
\end{theorem}
Equivalently, in the graph decomposition language, \cref{thm:det lem} states that if $m(G) \leq m_2(H_1,H_2)$, then $G$ can be edge-partitioned into an $H_1$-free graph and an $H_2$-free graph. 

In fact, we prove two more general graph decomposition theorems, \cref{thm:43,thm:pseudoforest} below, which we expect to be of independent interest. It is not hard to see (and we show this in the next section) that these two results, plus simple well-known arguments, suffice to prove \cref{thm:det lem}. 

Recall that a \emph{pseudoforest} is a graph in which every connected component contains at most one cycle. If $F$ is a subgraph of $G$, we denote by $G-F$ the subgraph of $G$ comprising all edges not in $F$.
\begin{theorem}\label{thm:pseudoforest}
    Every graph $G$ contains a pseudoforest $F \subseteq G$ such that $m_2(G-F)\le m(G)$. 
\end{theorem}
Equivalently, the conclusion of \cref{thm:pseudoforest} can be phrased as saying that $G$ can be edge-partitioned into a pseudoforest and a subgraph of bounded maximal 2-density.
We remark that many results along the lines of \cref{thm:pseudoforest} are known in the literature, starting with Hakimi's theorem \cite{MR0180501} from 1965, which states that $G$ can be decomposed into $k$ pseudoforests if and only if $m(G) \leq k$. More recently, several papers (e.g.\ \cite{2310.00931,MR4125897,MR3383251}) have proved strengthenings of Hakimi's theorem under the assumption that $m(G) \leq m$, where $m$ is not necessarily an integer. The novelty in \cref{thm:pseudoforest} (which does not follow from any of the results mentioned above) is that the condition we guarantee about $G-F$ is that it is sparse with respect to $m_2$, which is a \emph{different} density measure from the one we started with. We remark that \cref{thm:pseudoforest} is a slight weakening of a conjecture of Kuperwasser, Samotij, and Wigderson \cite[Conjecture 1.5]{kuper}; see \cref{sec:conclusion} for details.

On its own, \cref{thm:pseudoforest} already suffices to prove \cref{conj:KK} in almost all cases, namely for all tuples $(H_1,\dots,H_r)$ where $H_2$ contains a strictly $2$-balanced subgraph $H_2'$ with $m_2(H_2)=m_2(H_2')$ such that $H_2'$ is not a cycle; this includes almost all cases that were known before. However, \cref{thm:pseudoforest} cannot be used to resolve the remaining cases, so to prove \cref{thm:det lem} we need another decomposition result. To state it, let us define the \emph{maximal $\frac 43$-density} of $G$ to be
\[
m_{\frac 43}(G) \coloneqq \max \left\{ \frac{\e(J)}{\v(J) - \frac 43} : J \subseteq G, \v(J) \geq 2\right \}.
\]
\begin{theorem}\label{thm:43}
Let $m > \frac 32$ be a real number. Every graph $G$ with $m(G) \leq m$ contains a forest $F \subseteq G$ such that $m_{\frac 43}(G-F)<m$. 
\end{theorem}
We remark that the choice of $\frac 43$ is somewhat arbitrary, and our techniques can prove similar results where one replaces $\frac 43$ by another real number. However, working with the maximal $\frac 43$-density is useful for combining \cref{thm:pseudoforest,thm:43} and proving \cref{thm:det lem}.

The rest of this paper is organized as follows. In the next section, we show how to deduce \cref{thm:det lem} from \cref{thm:43,thm:pseudoforest}, and sketch the strategy we employ in the proofs of both \cref{thm:43,thm:pseudoforest}. We prove various general lemmas common to both proofs in \cref{sec:setup}, then prove \cref{thm:pseudoforest} in \cref{sec:pseudoforest}. The proof of \cref{thm:43} is more involved, and is split into \cref{sec:technical lemma}, where we state and prove a key technical lemma, and \cref{sec:43}, where we complete the proof. We end in \cref{sec:conclusion} with some concluding remarks and open problems.

\section{Proof overview}\label{sec:outline}
\subsection{Proof of Theorem \ref{thm:det lem} assuming Theorems \ref{thm:pseudoforest} and \ref{thm:43}}
In addition to the new ingredients, \cref{thm:pseudoforest,thm:43}, we will need three additional lemmas, all of which are well-known in the literature. The first is Nash-Williams' theorem \cite{MR0161333}, which was mentioned above. For its statement, let us define the \emph{maximal 1-density} (also known as the \emph{fractional arboricity}) of a graph $G$ to be
\[
m_1(G) \coloneqq \max \left \{ \frac{\e(J)}{\v(J) - 1} :J\subseteq G, \v(J) \geq 2\right\}.
\]
\begin{lemma}[Nash-Williams]\label{lem:NW}
    Let $k \geq 1$ be an integer. A graph $G$ can be edge-partitioned into $k$ forests if and only if $m_1(G) \leq k$.
\end{lemma}
The second lemma we need is a simple inequality relating $m(G)$ and $m_1(G)$; its short proof can be found, for example, in \cite[Lemma 5.4]{kuper}.
\begin{lemma}\label{lem:m1 m}
    For any graph $G$, we have $m_1(G) \leq m(G) +\frac 12$.
\end{lemma}
Finally, we quote a simple result proved independently in \cite[Proposition 5.2(a)]{kuper} and \cite[Lemma 1.9(ii)]{bowtell}, which establishes \cref{thm:det lem} whenever $H_2$ is not bipartite. 
\begin{lemma}\label{lem:bipartite}
    Let $H_1,H_2$ be graphs with $m_2(H_1) > m_2(H_2)>1$. If $H_2$ is strictly $2$-balanced and not bipartite, then any graph $G$ which is Ramsey for $(H_1,H_2)$ satisfies $m(G) > m_2(H_1,H_2)$.
\end{lemma}
With these preliminaries, we are ready to prove \cref{thm:det lem}.
\begin{proof}[Proof of \cref{thm:det lem}]
    Fix a pair of graphs $H_1,H_2$ with $m_2(H_1) > m_2(H_2) > 1$. Recall that this implies that $m_2(H_1) > m_2(H_1,H_2)>m_2(H_2)$. We may assume without loss of generality that $H_2$ is strictly $2$-balanced; indeed, if it is not, we pass to a strictly $2$-balanced subgraph $H_2' \subseteq H_2$ with $m_2(H_2') = m_2(H_2)$, and observe that any $H_2'$-free graph is necessarily $H_2$-free. Thus, we henceforth assume that $H_2$ is strictly $2$-balanced. This in particular implies that $H_2$ is connected and that the minimum degree of $H_2$ is at least $2$. (In fact, any strictly $2$-balanced graph is $2$-connected as shown in \cite[Lemma 3.3]{rajkosteger}.)

    Let us fix a graph $G$ with $m(G) \leq m_2(H_1,H_2)$. Our goal is to edge-partition $G$ into an $H_1$-free graph and an $H_2$-free graph. We are done by \cref{lem:bipartite} if $H_2$ is not bipartite, so let us assume that $H_2$ is bipartite.

    If $m(G) \leq \frac 32$, then $m_1(G) \leq 2$ by \cref{lem:m1 m}, and therefore \cref{lem:NW} implies that $G$ can be edge-partitioned into two forests. As $m_2(H_1) > m_2(H_2)>1$, neither $H_1$ nor $H_2$ is a forest, and thus this is an edge-partition of $G$ into an $H_1$-free graph and an $H_2$-free graph.

    Now suppose that $H_2$ contains at least two cycles. As $H_2$ is connected, this implies that $H_2$ is not contained in any pseudoforest. We now apply \cref{thm:pseudoforest}  to find a pseudoforest $F \subseteq G$ with $m_2(G-F) \leq m(G)\leq m_2(H_1,H_2)<m_2(H_1)$. By the above, $F$ is $H_2$-free. Moreover, as $m_2(G-F) < m_2(H_1)$, we see that $G-F$ is $H_1$-free. This is our desired decomposition.

    Finally, we may assume that $m(G) > \frac 32$, and that $H_2$ is connected, bipartite, and contains only one cycle. As the minimum degree of $H_2$ is at least $2$, we find that $H_2=C_\ell$ is a cycle, for some even $\ell \geq 4$. As $m_2(C_\ell) = (\ell-1)/(\ell-2)$, we conclude that $m_2(H_2) \leq m_2(C_4) = \frac 32$. Therefore,
    \begin{align*}
    m_2(H_1,H_2) &= \max_{J \subseteq H_1} \frac{\e(J)}{\v(J) - 2 + 1/m_2(H_2)}\\
    &\leq \max_{J \subseteq H_1} \frac{\e(J)}{\v(J) - 2 + 1/(\frac 32)}\\
    &=\max_{J \subseteq H_1} \frac{\e(J)}{\v(J) - \frac 43} = m_{\frac 43}(H_1).
    \end{align*}
    We now apply \cref{thm:43} with $m = m(G) > \frac 32$. We find a forest $F \subseteq G$ with $m_{\frac 43}(G-F) < m$. As $H_2$ is not a forest, we see that $F$ is $H_2$-free, whereas the above implies that $m_{\frac 43}(G-F) < m \leq m_2(H_1,H_2) \leq m_{\frac 43}(H_1)$, and thus $G-F$ is $H_1$-free. This is the desired decomposition, and the proof is complete.
\end{proof}
\subsection{Proof strategy and sketch}
We now sketch the strategy we employ to prove \cref{thm:pseudoforest,thm:43}. The same general strategy is used for both; we begin with \cref{thm:pseudoforest}, which is somewhat simpler both conceptually and technically.

So let us fix some real number $m >0$ and some graph $G$ with $m(G) \leq m$. Our goal is to define a pseudoforest $F \subseteq G$, in such a way that we can control $m_2(G-F)$. Our main tool to do this is the concept of \emph{allocations}, which are a fractional version of an edge-orientation of $G$.

\begin{definition}
Let $G$ be a graph and $m\ge 0$ a real number. An $m$-\emph{allocation} of $G$ is a mapping $\theta:V(G)^2\rightarrow [0,1]$ such that the following hold. 
\begin{enumerate}[ref=(\arabic*)]
    \item If $u, v \in V(G)$ and $uv \notin E(G)$, then $\theta(u,v)=\theta(v,u)=0$.
    \item $\theta(u,v)+\theta(v,u)=1$ for every edge $uv \in E(G)$.  
    \item\label{it: outdegree bound} For every vertex $u \in V(G)$, we have $\sum_{v \in V(G)}{\theta(u,v)}\le m$. 
\end{enumerate}
\end{definition}
Note that if $\theta$ is integer-valued, it can be viewed as an edge-orientation of $G$ (namely we orient an edge $uv$ according to whether $\theta(u,v)$ is $0$ or $1$). If $\theta$ is integer-valued, then condition \ref{it: outdegree bound} above corresponds to saying that every vertex has outdegree at most $m$. In general, allocations can be thought of as fractional orientations, with an upper bound on the ``fractional outdegree'' of every vertex.
The reason we care about allocations is that the existence of an $m$-allocation of $G$ is equivalent to the statement $m(G) \leq m$, as stated in the next theorem. 

\begin{theorem}\label{thm:allocations}
Let $G$ be a graph and $m>0$ a real number. Then $m(G)\le m$ if and only if $G$ admits an $m$-allocation.
\end{theorem}
Results of this type go back at least to work of Hakimi \cite{MR0180501} and of Frank and Gy\'arf\'as \cite{MR0519276}, who proved \cref{thm:allocations} under the extra assumption that $m$ is an integer (in which case they can get a true orientation, rather than an allocation). Their result actually implies \cref{thm:allocations} by passing to an appropriate auxiliary multigraph; one can also deduce \cref{thm:allocations} from the results of \cite{MR1080823}. We present an alternative proof, which is closely modeled on the proof of \cite[Lemma 2]{MR4396478}, in \cref{sec:allocations proof}. For the moment we continue with the proof overview.

As we are given a graph $G$ with $m(G) \leq m$, we apply \cref{thm:allocations} to find an $m$-allocation $\theta:V(G)^2 \to [0,1]$.
We say that an edge $e=uv\in E(G)$ is \emph{integral} if $\theta(u,v),\theta(v,u) \in \{0,1\}$. Otherwise, if $\theta(u,v), \theta(v,u) \in (0,1)$, then we say that the edge $uv$ is \emph{fractional}. We denote by $E_{\text{int}}(G,\theta)$ and $E_{\text{frac}}(G,\theta)$ the set of integer and fractional edges of $G$, respectively. Moreover, we can naturally associate with $\theta$ a digraph $D_\theta=(V(G),A_\theta)$, where $A_\theta\coloneqq \{(u,v):\theta(u,v)=1\}$. Clearly, $D_\theta$ is an orientation of the integral edges in $G$. Similarly, we denote by $G^{\text{frac}}_\theta$ the subgraph of $G$ consisting of the fractional edges $E_{\text{frac}}(G,\theta)$.

There are a number of elementary operations one can perform on an $m$-allocation to get another $m$-allocation. By doing such operations, we may assume that $\theta$ satisfies certain additional properties; for example, we can ensure that $G_{\text{frac}}^\theta$ is a forest (see \cref{lem: forest}).

Having chosen $\theta$ so that it satisfies these extra properties, we are ready to define the pseudoforest $F$ that we need to output in order to prove \cref{thm:pseudoforest}. To do so, for every vertex $u \in V(G)$ that is the source of at least one arc in $D_\theta$, we pick one such arc $(u,v)$ arbitrarily, and define $\overrightarrow F$ to consist of all chosen arcs. Thus, $\overrightarrow F$ is a digraph with maximum outdegree $1$, so its underlying graph $F \subseteq G$ is a pseudoforest.

Having defined $F$, it remains to prove the main claim in \cref{thm:pseudoforest}, namely that $m_2(G-F) \leq m$. Equivalently, we wish to prove that for every $U \subseteq V(G)$ with $|U| \geq 3$, we have $\e_{G-F}(U) -1 \leq m(|U|-2)$, where we use the notation $\e_{G-F}(U) \coloneqq \e((G-F)[U])$. So we assume for contradiction that this is not the case, and proceed to analyze the structure of a violating set $U$. The proof is somewhat involved and we defer the details to \cref{sec:pseudoforest}, but the basic idea is to examine how $U$ interacts with the digraph $D_\theta$, and to use both the definition of $F$ and the extra properties we imposed on the $m$-allocation $\theta$.

The proof of \cref{thm:43} follows a similar strategy, but is more involved for a few different reasons. The first reason is that in \cref{thm:43}, we wish to output a \emph{forest} $F \subseteq G$. So we cannot simply pick an out-arc in $D_\theta$ for every vertex of $G$, as this creates cycles. However, by examining the structure of the digraph $D_\theta$ (specifically, by partitioning it into strongly connected components), we can pick an out-arc for \emph{almost} every vertex, and control the set of vertices with no out-arcs, while ensuring that the subgraph we so define is a forest. Having done this, we again assume for contradiction that $m_{\frac 43}(G-F) \geq m$, and work with a violating set $U\subseteq V(G)$, namely a set $U$ with $|U|\geq 2$ satisfying $\e_{G-F}(U) \geq m(|U|-\frac 43)$. The analysis of $U$ in this proof is substantially more complicated, and we need to split into a number of different cases, depending on the intersection pattern of $U$ with the strongly connected components of $D_\theta$. Most of these cases can be dealt with via rather simple (but tedious) arguments; this is done in \cref{sec:technical lemma}, where we state and prove a technical lemma dealing with most of these cases. For the remaining cases, some more care needs to be taken in the definition of $\theta$ and the choice of $F$; by imposing some further conditions on $\theta$ and $F$, we are able to dispense with the remaining cases in \cref{sec:43}.

\section{Setup and general lemmas}\label{sec:setup}
In this section, we state and prove various general lemmas that we will use in the proofs of both \cref{thm:43,thm:pseudoforest}.
\subsection{Proof of Theorem \ref{thm:allocations}}\label{sec:allocations proof}
We begin with the proof of \cref{thm:allocations} (which we deferred in \cref{sec:outline}), following the proof approach of \cite[Lemma 2]{MR4396478}. As discussed above, there are a number of other ways of proving this result, including a reduction to \cite[Theorem 1]{MR0519276} or an application of Hall's matching theorem.
\begin{proof}[Proof of \cref{thm:allocations}]
    Associated to the graph $G$ and the parameter $m$, we define a flow network $F$ as follows. $F$ has a single source vertex $s$ and a single sink vertex $t$, as well as two other sets of vertices $V,E$. The vertices in $V$ are identified with $V(G)$, and the vertices in $E$ are identified with $E(G)$. There are directed edges $s \to V \to E \to t$, defined as follows. For every $v \in V$, we place a directed edge of capacity $m$ from $s$ to $v$. For every $v \in V$ and $e \in E$ such that $v$ is incident to $e$ in $G$, we place an edge of infinite capacity from $v$ to $e$. Finally, for every $e \in E$, we place a directed edge of capacity $1$ from $e$ to $t$.

    The max-flow min-cut theorem implies that the maximum $s-t$ flow in $F$ equals the minimum weight of an $s-t$ cut in $F$. As all edges between $V$ and $E$ have infinite capacity, they do not participate in any minimum cut. Hence any minimum cut must consist of some edges between $s$ and $V$, and some edges between $E$ and $t$. Fix such a minimum cut, and let $A \subseteq V, B \subseteq E$ be the set of vertices in $V,E$, respectively, incident to a cut edge. The weight of this cut is $m|A| + |B|$. For this to be a true $s-t$ cut in $F$, we must have no directed edges from $V \setminus A$ to $E \setminus B$; by the definition of the edges $V \to E$, this implies that $E \setminus B$ consists only of edges with both endpoints in $A$. Therefore, $|E \setminus B| \leq \e(G[A])$, and thus $|B| \geq \e(G) - \e(G[A])$. Therefore the weight of the cut is at least $m|A| + \e(G) -\e(G[A])$. Conversely, we see that for any $A \subseteq V$, there is a cut of weight exactly $m|A| + \e(G) - \e(G[A])$, obtained by defining $E \setminus B$ to be the set of edges entirely contained in $A$.

    Suppose first that $m(G) \leq m$. In particular, we find that $\frac{\e(G[A])}{|A|} \leq m$, implying that the weight of the minimum cut is at least $\e(G)$. The cut separating $t$ from the other vertices has weight $\e(G)$, as there are $\e(G)$ edges of capacity $1$ incident to $t$, and is thus a minimum cut. Therefore, there exists an $s-t$ flow in $F$ of weight exactly $\e(G)$. In such a flow every vertex in $E$ must receive weight exactly $1$. This means that for any $e=uv \in E(G)$, we have a total flow of $1$ into the node $e$, coming from its neighbours $u,v \in V$. Defining $\theta(u,v)$ to be the flow $u$ sends to $e$, and similarly $\theta(v,u)$ the flow $v$ sends to $e$, we find that $\theta(u,v)+\theta(v,u)=1$. Moreover, the fact that every $u \in V$ is incident to an edge of capacity $m$ from $s$ shows that this defines a valid $m$-allocation.

    Conversely, suppose that there is an $m$-allocation of $G$. As in the previous paragraph, this defines an $s-t$ flow of weight $\e(G)$ in $F$. Thus, every cut must have weight at least $\e(G)$. Therefore, for any $A \subseteq V$, we have that $m|A| + \e(G) - \e(G[A]) \geq \e(G)$, implying that $m(G) \leq m$.
\end{proof}

\subsection{Strong components in digraphs}
Let us fix the following terminology. For a directed graph $D$, we denote by $A(D)$ the set of arcs (directed edges) in $D$. Given a directed graph $D$, a \emph{strong component} of $D$ is an inclusion-wise maximal subset $X \subseteq V$ of vertices such that the induced subdigraph $D[X]$ is strongly connected. It is well-known that the strong components of a digraph $D$ partition its vertex-set. We call a strong component $X$ of $D$ a \emph{terminal component} if there are no arcs in $D$ that start in a vertex in $X$ and end in a vertex outside of $X$. We will need the following statements about how changing a digraph affects the number and structure of its terminal components.
\begin{lemma}\label{lem:components}
\hfill
    \begin{lemenum}
        \item\label{it:subgraph} Let $D_1$ and $D_2$ be digraphs on the same vertex-set such that $A(D_1)\subseteq A(D_2)$. Then the number of terminal components of $D_2$ is at most as large as the number of terminal components of $D_1$. Furthermore, if there exist distinct terminal components $U, W$ of $D_1$ and a directed path in $D_2$ that starts in $U$ and ends in $W$, then the number of terminal components of $D_2$ is strictly smaller than the number of terminal components of $D_1$. 
        \item\label{it:delete arc} Let $X$ be a terminal component of a digraph $D$ and let $(u,v) \in A(D[X])$. Then the collection of terminal components of $D-(u,v)$ is obtained from the collection of terminal components of $D$ by replacing $X$ with a subset of it that contains $u$.
        \item\label{it:outside terminal} Let $D$ be a digraph and let $u$ be a vertex of $D$ that does not belong to any terminal component. If there exists a directed path in $D$ that starts in $u$, ends in a terminal component of $D$ and does not use the arc $(u,v)$, then the digraphs $D-(u,v)$ and $D$ have the same terminal components. 
        \item\label{item:adduselessedge} Let $D$ be a digraph and let $u$ be a vertex that does not belong to any terminal component. Let $D+(u,v)$ be the digraph obtained by adding a new arc $(u,v)$ to $D$. Then $D$ and $D+(u,v)$ have the same terminal components.
        \item\label{it:singleton} Let $D$ be a digraph and $\{u\}$ a singleton terminal component of $D$, and let $D+(u,v)$ be obtained from $D$ by adding a new arc $(u,v)$ to $D$. Then one of the following holds.
        \begin{itemize}
            \item The number of terminal components of $D+(u,v)$ is smaller than that of $D$.
            \item The collection of terminal components of $D+(u,v)$ is obtained from the collection of terminal components of $D$ by replacing $\{u\}$ with a terminal component containing $\{u,v\}$.
        \end{itemize}
    \end{lemenum}
\end{lemma}
\begin{proof}\hfill
\begin{enumerate}[label=(\alph*)]
    \item\label{proofit:path} For every strong component $X$ of $D_1$, we have that $D_1[X]$ and thus $D_2[X]$ is strongly connected, and hence $X$ is fully contained in one of the strong components of $D_2$. This directly implies that every strong component of $D_2$ is a disjoint union of some strong components of $D_1$. Now let $Y$ be any terminal component of $D_2$. Let $X_1,\ldots,X_k$ be the strong components of $D_1$ such that $Y=\bigcup_{i=1}^{k}{X_i}$. Consider the digraph $D_1[Y]$, whose strong components are $X_1,\ldots,X_k$. It is well-known (and easy to see) that every digraph contains at least one terminal component. Hence, there exists $i \in \{1,\ldots,k\}$ such that there are no edges in $D_1[Y]$ that leave $X_i$. But since $Y$ is a terminal component of $D_2\supseteq D_1$, this means that there are also no edges in the whole of $D_1$ that leave $X_i$, and hence $X_i$ is a terminal component of $D_1$ that is contained in $Y$. 
    
    We have shown that every terminal component of $D_2$ contains a terminal component of $D_1$, and hence the number of terminal components of $D_2$ is at most the number of terminal components of $D_1$. 
    
   Let $P$ be a directed path in $D_2$ that starts in $U$ and ends in $W$. We claim that either $U$ and $W$ are contained in the same terminal component of $D_2$, or at least one of $U, W$ is not contained in any terminal component of $D_2$. In both cases, we can immediately conclude that there must be strictly more terminal components in $D_1$ than in $D_2$, as desired.
    
    Suppose that contrary to the above claim, there exist two distinct terminal components $Y_1, Y_2$ of $D_2$ such that $U\subseteq Y_1$ and $W\subseteq Y_2$. Since $P$ starts in $U$ and ends in $W$, this implies that there exists an arc in $D_2$ leaving $Y_1$, a contradiction to it being a terminal component. This concludes the proof. 
    \item\label{proofit:delete edge} Let $Y$ be a terminal component of $D-(u,v)$. If $u \notin Y$, then $Y$ is a terminal component of $D$ that is disjoint from $X$. So suppose $u \in Y$. Since every other vertex in $Y$ is reachable via a directed path from $u$ in $D-(u,v)$, and since no arc in $D-(u,v)$ leaves $Y$ and no arc in $D$ leaves $X$, it follows that $Y$ is equal to the set of vertices reachable in $D-(u,v)$ via a directed path starting in $u$, and that $Y\subseteq X$. In fact, it can be easily checked that the latter set of vertices always induces a strongly connected subdigraph of $D-(u,v)$ (since every vertex in $X$ can reach $u$ via a directed path in $D-(u,v)$). Hence, we have shown that the set of vertices reachable from $u$ in $D-(u,v)$ forms the unique terminal component of $D-(u,v)$ contained in $X$.

    In the other direction, note that any terminal component of $D$ distinct from $X$ is also disjoint from $X$ and hence remains a terminal component also in $D-(u,v)$. This shows that the terminal components of $D$ and $D-(u,v)$ are the same, apart from the replacement of $X$ by a subset containing $u$.
    \item Let $X$ be a terminal component of $D-(u,v)$. Unless $u \in X$ and $v \notin X$, also in $D$ there are no arcs leaving $X$ and thus $X$ forms a terminal component also of $D$. Now suppose $u \in X$ and $v \notin X$. Then $(u,v)$ is the only arc in $D$ leaving $X$. Let $P$ be a directed path in $D-(u,v)$ starting in $u$ and ending in a terminal component $Y$ of $D$. Then, since $P$ cannot leave $X$, we have $V(P)\subseteq X$ and thus $X\cap Y\neq \emptyset$. Since $D[X]$ is strongly connected and $Y$ is a strong component of $D$, by definition, we have $X\subseteq Y$. In particular, we have $u \in Y$, which contradicts our assumption that $u$ is not part of any terminal component of $D$. This contradiction shows that our assumption above was wrong, and hence $X$ is indeed a terminal component also of $D$.

    In the other direction, suppose $X$ is some terminal component of $D$. 
    We assumed that $u$ is not contained in any terminal component of $D$, hence $u \notin X$. But then necessarily $(u,v) \notin A(D[X])$, and thus $X$ is also a terminal component of $D-(u,v)$.
    \item Let $X$ be a terminal component of $D$. Then by assumption $u \notin X$, and hence also in $D+(u,v)$ no arc leaves $X$. Hence, $X$ is also a terminal component of $D+(u,v)$. In the other direction, suppose that $X$ is a terminal component of $D+(u,v)$. If $(u,v)\notin A(D[X])$, then clearly $X$ is also a terminal component of $D$. So now suppose that $(u,v) \in A(D[X])$. We can now apply item \ref{proofit:delete edge} of this lemma to find that there exists a terminal component of $D+(u,v)-(u,v)=D$ contained in $X$ that contains $u$. This is a contradiction, since we assumed that $u$ is not part of any terminal component of $D$. Hence our assumption was false, we indeed have $(u,v)\notin A(D[X])$, and hence $X$ is also a terminal component of $D$.
    \item Suppose that the number of terminal components of $D+(u,v)$ is not smaller than that of $D$. Together with the fact that every strongly connected (terminal) component of $D+(u,v)$, which does not contain $u$, is also a strongly connected (terminal) component of $D$, we get that the collection of terminal components of $D+(u,v)$ is obtained from the collection of terminal components of $D$ by replacing $\{u\}$ with some $X\supseteq\{u\}$. By definition of a terminal component, we find that $X$ contains $v$ as well.  
    \qedhere
    
    %Suppose first that there exists a directed path from $v$ to a terminal component of $D$ distinct from $\{u\}$. Attaching the arc $(u,v)$ to such a path then provides a directed path in $D+(u,v)$ between two distinct terminal components of $D$, and hence $D+(u,v)$ has strictly fewer terminal components than $D$ by item \ref{proofit:path} of this lemma. 
    
    %Next, suppose that such a path does not exist. Let $A$ be the set of vertices in $D$ reachable from $v$ in $D$ (or equivalently, in $D+(u,v)$). By assumption $A$ does not intersect a terminal component of $D$ distinct from $\{u\}$. However, for every vertex $a \in A$ there exists a directed path in $D$ (and thus in $D[A])$ starting in $a$ and ending in a terminal component of $D$, and this terminal component must be equal to $\{u\}$. Hence, every vertex in $A$ reaches $u$ in $D[A]$, and it follows that $(D+(u,v))[A]$ is strongly connected. Since there are no edges in $D+(u,v)$ leaving $A$, it is in fact a terminal component of $D+(u,v)$. 
    %Moreover, a set $X$ of vertices forms a terminal component in $D$ disjoint from $\{u\}$ if and only if it forms a terminal component of $D+(u,v)$ disjoint from $A$. Hence, the digraph $D+(u,v)$ has the same terminal components as $D$, except that the singleton terminal component $\{u\}$ in $D$ is replaced by the terminal component $A\supseteq \{u,v\}$ in $D+(u,v)$. 
    \end{enumerate}
\end{proof}

\subsection{Properties of and operations on allocations}

A natural way of changing a given $m$-allocation $\theta$ of a graph to another is to ``shift'' weights along a cycle. As this operation will be used in several places of our proof, we isolate the elementary facts about it in the following statement.

\begin{observation}\label{obs:shift}
Let $G$ be a graph, let $m>0$ be a real number, and let $\theta:V(G)^2\rightarrow \mathbb{R}_+$ be an $m$-allocation on $G$. Let $C$ be a cycle in $G$, and let $u_0,u_1,u_2,\ldots,u_k=u_0$ for some $k \ge 3$ be the cyclic sequence of vertices along $C$. Let $$\varepsilon\coloneqq\min\{\theta(u_{i-1},u_{i}):1 \leq i \leq k\}.$$ Suppose that $\varepsilon>0$, and define a new mapping $\theta':V(G)^2\rightarrow [0,1]$ by setting
    $$\theta'(x,y)\coloneqq \begin{cases}
    \theta(x,y)-\varepsilon, &  \text{ if }\exists i\in \{1,\ldots,k\}\text{ such that }(x,y)=(u_{i-1},u_i), \\
    \theta(x,y)+\varepsilon, &  \text{ if }\exists i\in \{1,\ldots,k\}\text{ such that }(x,y)=(u_{i},u_{i-1}), \\
    \theta(x,y), & \text{ otherwise.}
    \end{cases}$$ for every $x,y \in V(G)$. Then $\theta'$ is still an $m$-allocation of $G$. Furthermore, $D_{\theta'}$ is obtained from $D_\theta$ by removing the arcs $$\{(u_{i-1},u_i):1\le i \le k, \theta(u_{i-1},u_i)=1\}$$ and adding the arcs $$\{(u_i,u_{i-1}):1\le i \le k, \theta(u_{i-1},u_i)=\varepsilon\}.$$ 
\end{observation}

Given a graph $G$, a cycle $C$ equipped with a cyclic orientation and an $m$-allocation $\theta$ on $G$, the above forms a well-defined way to generate a new $m$-allocation $\theta'$ on $G$, which we from now on will refer to as the \emph{$m$-allocation obtained from $\theta$ by shifting along $C$}.

In our proofs, we will work with special allocations which are chosen to have certain extra properties. The following definition captures the way we choose these special allocations
\begin{definition}
Given a real number $m$ and a graph $G$ for which there exists an $m$-allocation, we say that an $m$-allocation is \emph{optimal} if it is chosen such that:
\begin{enumerate}
    \item The number of terminal components of $D_\theta$ is minimum among all possible choices of $m$-allocations $\theta$, and
    \item subject to point (1), the number of fractional edges is minimized (equivalently, the number of integral edges is maximized), and
    \item subject to points (1) and (2), the total number of vertices contained in terminal components is maximized, and
    \item subject to points (1), (2) and (3), the number of arcs in $D_\theta$ that end in a vertex forming a singleton terminal component of $D_\theta$ is maximized.
\end{enumerate}
\end{definition}
The following lemmas establish some properties that follow from this choice of $\theta$. Recall that $G_\theta^{\mathrm{frac}}$ is the spanning subgraph of $G$ comprising all the fractional edges.
\medskip
\begin{lemma}\label{lem: forest} If $\theta$ is an optimal $m$-allocation, then $G_\theta^{\mathrm{frac}}$ is a forest. \end{lemma}
\begin{proof}
    Suppose for a contradiction that there exists a cycle $C$ in $G_\theta^{\mathrm{frac}}$, and let $u_0,u_1,\ldots,u_k=u_0$ be its cyclic sequence of vertices. Let $\theta'$ be the $m$-allocation obtained from $\theta$ by shifting along $C$. By \cref{obs:shift}, $\theta'$ is still an $m$-allocation on $G$. Let $i_0\in \{1,\ldots,k\}$ be an index that minimizes $\theta(u_{i_0-1},u_{i_0})$. By \cref{obs:shift} and since all edges along $C$ are fractional in $\theta$, we have $A(D_{\theta'}) \supseteq A(D_{\theta})\cup \{(u_{i_0},u_{i_0-1})\}$. This implies by \cref{it:subgraph} that the number of terminal components in $D_{\theta'}$ is at most as large as the number of terminal components in $D_\theta$. Together with the fact that $|E_{\mathrm{frac}}(G,\theta')|<|E_{\mathrm{frac}}(G,\theta)|$ (since $A(D_{\theta'}) \supsetneq A(D_{\theta})$) this contradicts the optimality of $\theta$. Hence our assumption was wrong, and the graph $G_\theta^\mathrm{frac}$ is indeed a forest.
\end{proof}
\begin{lemma}
\label{lemma: no triangle between terminal components}
Let $\theta$ be an optimal $m$-allocation, let $v,w$ be vertices of a terminal component $C$ of $D_\theta$ such that $(v,w) \in A(D_\theta)$, and let $P=(w=u_0,\dots, u_\ell=v)$ be a path in $G_\theta^{\mathrm{frac}}$ such that all the inner vertices are in some terminal component of $D_\theta$ different from $C$. Then there exists $0\leq i\leq\ell-1$ such that $u_{i}$ and $u_{i+1}$ lie in the same terminal component of $D_\theta$.

In particular, if $U$ and $W$ are distinct terminal components of $D_\theta$, then for every $u \in U$ the set $N_G(u)\cap W$ is independent in $G$. 
\end{lemma}
\begin{proof}
Suppose for contradiction that for all $0\leq i\leq \ell-1$, it holds that $u_i$ and $u_{i+1}$ are in different terminal components of $D_\theta$. Let $\theta'$ be the $m$-allocation on $G$ that is obtained from $\theta$ by shifting along the cycle obtained by adding $(v,w)$ to $P$. By \cref{obs:shift} and since by assumption $\theta(v,w)=1$, $A(D_{\theta'})$ is obtained from $A(D_\theta)$ by removing the arc $(v,w)$ and adding at least one of the arcs $(u_i,u_{i-1})$ for some $1\leq i\leq \ell$. Let $1\leq j\leq \ell$ be such that $(u_j,u_{j-1})$ is added. Let $D' = D_\theta-(v,w)$. By \cref{it:delete arc}, it follows that the terminal components of $D'$ are the same as the terminal components of $D_\theta$, except that $C$ is replaced by some $X \subseteq C$ with $v \in X$. Observe that there exists a path $Q$ from $w$ to $v$ in $D'$ since $D_\theta[C]$ is strongly connected and $(v,w)$ is not part of any path from $w$ to $v$. 

Recall that by assumption, if $2\le j\leq \ell -1 $ then $u_j$ and $u_{j-1}$ are in different terminal components of $D_\theta$, and these terminal components are distinct from $C$. Hence they also lie in distinct terminal components of $D'$. If $j=\ell$, then $u_j=u_{\ell}=v$ is contained in the terminal component $X\subseteq C$ of $D'$, and $u_{j-1}=u_{\ell-1}$ is in a terminal component of $D_\theta$ distinct from $C$. Hence, $u_j,u_{j-1}$ are also in different terminal components in this case. Finally, if $j=1$, then adding the arc $(u_1,u_{0}=w)$ introduces a path (namely $(u_1,w)+Q$) from $u_1$ to $v$ in $D'$, and $u_1$ and $v$ are also in different terminal components of $D'$. In each case, we conclude by \cref{it:subgraph} that $D'+(u_j,u_{j-1})$ and thus also $D_{\theta'}\supseteq D'+(u_j,u_{j-1})$ has strictly fewer terminal components than $D_\theta$, contradicting the optimality of $\theta$.

For the second statement, note first that all edges between $u$ and $W$ must be fractional, as both $U$ and $W$ are terminal components of $D_\theta$. Therefore, $N_G(u) \cap W$ cannot contain any fractional edges, since this would produce a triangle in $G_\theta^{\mathrm{frac}}$, contradicting \cref{lem: forest}. On the other hand, if $N_G(u) \cap W$ contains an integral edge $(v,w) \in A(D_\theta)$, we obtain a contradiction to the first statement of this lemma by setting $P = (w, u, v)$. 
\end{proof}

\subsection{Spines}
In order to define the pseudoforest and forest guaranteed in \cref{thm:pseudoforest,thm:43}, respectively, we use the notion of a \emph{spine} of a digraph, which we now define.
\begin{definition}\label{def:spine}
Let $D$ be a digraph. A spanning subdigraph $H\subseteq D$ is called a \emph{spine} of $D$ if the following hold.
\begin{itemize}
    \item Every vertex in a non-terminal component of $D$ has out-degree exactly $1$ in $H$.
    \item In every terminal component of $D$, all vertices have out-degree $1$ in $H$, apart from at most one vertex of out-degree $0$.
    \item The underlying graph of $H$ is a forest.
\end{itemize}
\end{definition}

\begin{lemma}\label{lemma:forestexist} Let $D$ be a digraph and $R\subseteq V(D)$ a set of vertices containing exactly one vertex from each terminal component of $D$. Then $D$ has a spine $H$ such that every vertex in $V(D)\setminus R$ has out-degree exactly one in $H$. \end{lemma}
\begin{proof}
Let $X$ be any strong component of $D$. Then $D[X]$ is strongly connected, and hence for every vertex $x \in X$ it contains an in-tree with root vertex $x$ that contains all the vertices in $X$. 

To build $H$, do the following. For every non-terminal component $X$, pick an arc $(x,y) \in A(D)$ such that $x \in X, y \notin X$ and a spanning in-tree $T_X$ of $D[X]$ rooted at $x$. Then add all the arcs of $T_X$ as well as the arc $(x,y)$ to the arc-set of $H$. Next, for every terminal component $X$, pick the unique vertex $r\in R \cap X$ and an in-tree rooted at $r$ that spans $D[X]$ and add all the arcs of this tree to $H$. 

At the end of this process, $H$ defines a spanning subdigraph of $D$ in which every vertex has exactly one out-arc, except for one vertex in each terminal component (namely, the vertices in $R$). 
Finally, suppose towards a contradiction that the underlying graph $F$ of $H$ contains a cycle. Then since $H$ has maximum out-degree $1$, this cycle would have to be directed in $H$, and thus it would be fully contained in a strong component of $D$. However, it follows directly by definition of $H$ that $F$ restricted to any strong component of $D_\theta$ is a tree, yielding the desired contradiction. Thus, $F$ is indeed a forest, as desired. 
\end{proof}
\begin{lemma}\label{lemma: no_edge}
Let $\theta$ be an optimal $m$-allocation, and let $H$ be a digraph containing a spine of $D_\theta$. Let further $u,v,w$ be vertices such that $\{u\}$ is a singleton terminal component of $D_\theta$, $uv, uw \in E(G_\theta^{\mathrm{frac}})$, and $(v,w) \in A(D_\theta)\setminus A(H)$. Then $\{w\}$ is also a singleton terminal component of $D_\theta$.
\end{lemma}
\begin{proof}
Towards a contradiction, suppose that $\{w\}$ is not a singleton terminal component of $D_\theta$.

Let $\theta'$ be the $m$-allocation on $G$ that is obtained from $\theta$ by shifting along the triangle $u,v,w$ (in this cyclic order). By \cref{obs:shift}, $A(D_{\theta'})$ is obtained from $A(D_\theta)$ by removing the arc $(v,w)$ and adding one (or both) of the arcs $(v,u),(u,w)$. 

If $v$ is contained in any terminal component of $D_\theta$, then the same terminal component of $D_\theta$ also must contain $w$, and hence we are in the setting of \cref{lemma: no triangle between terminal components}, which yields a contradiction to this case (note that since $u$ is contained in a singleton terminal component, it is necessarily contained in a distinct terminal component from $v,w$). 

So suppose in the following that $v$ is not contained in any terminal component of $D_\theta$. Let $P$ be the directed path in the spine of $D_\theta$ contained in $H$ that one obtains by starting at $v$ and repeatedly following the unique out-arc in the spine until one reaches a terminal component of $D_\theta$. Since we assumed $(v,w)\notin A(H)$, the path $P$ does not use $(v,w)$, and so we can apply \cref{it:outside terminal} to find that the digraphs $D_\theta-(v,w)$ and $D_\theta$ have the same terminal components. Since $D_{\theta'} \supseteq D_\theta-(v,w)$, by \cref{it:subgraph}, we find that $D_{\theta'}$ has at most as many terminal components as $D_\theta$. Also, since it is obtained from $D_\theta-(v,w)$ by adding at least one of $(v,u), (u,w)$, it has at least as many arcs as $D_\theta$. If both of the arcs are added, then $\theta'$ induces strictly more integral edges than $\theta$, which contradicts the optimality of $\theta$. Thus, in the following suppose that exactly one of the arcs $(u,w),(v,u)$ is added.

For a first case, suppose that $(u,w) \in A(D_{\theta'})$. It then follows from \cref{it:singleton} that one of the following holds.

\begin{itemize}
    \item $D_{\theta'}$ has strictly fewer terminal components than $D_\theta-(v,w)$ (and thus than $D_\theta$).
    \item In $D_{\theta'}$ we have exactly as many terminal components as in $D_\theta-(v,w)$, but strictly more vertices that are contained in terminal components than in $D_\theta-(v,w)$ (and thus than in $D_\theta$). 
\end{itemize}
In both cases, this yields a contradiction to the optimality of $\theta$. 

For the second case, suppose that $(v,u) \in A(D_{\theta'})$ and $(u,w) \notin A(D_{\theta'})$. By what was said above and \cref{it:subgraph}, we have that $D_{\theta'}$ has at most as many terminal components as $D_\theta$, and at least as many integral edges as $D_\theta$. Furthermore, since we assumed that $v$ is not part of any terminal component of $D_\theta$, it is also not part of any terminal component of $D_\theta-(v,w)$, since we found above that these two digraphs have the same terminal components. Hence, \cref{item:adduselessedge} implies that also the digraph $D_{\theta'}=(D_\theta-(v,w))+(v,u)$ has the same terminal components as $D_\theta$, and hence the number of vertices that are in terminal components is the same in $D_\theta$ and $D_{\theta'}$. However, we claim that the number of edges of $D_{\theta'}$ ending in singleton terminal components is strictly larger than in $D_\theta$, which will yield the desired contradiction to the optimality of $\theta$. To see this, note that by what was said above the set of singleton terminal components of $D_\theta$ and $D_\theta'$ is identical. However, adding the edge $(v,u)$ increases the in-degree of the singleton terminal component $u$, and no in-degree of any other singleton terminal component is reduced, since we assumed that $w$ is not a singleton terminal component. This contradiction concludes the proof. 
\end{proof}
\section{Proof of Theorem \ref{thm:pseudoforest}}\label{sec:pseudoforest}
Given these preliminaries, we are now ready to prove \cref{thm:pseudoforest}.
\begin{proof}[Proof of \cref{thm:pseudoforest}]
Let $m = m(G)$ and, by \cref{thm:allocations}, let $\theta$ be an optimal $m$-allocation of $G$. Let $H$ be a spine of $D_\theta$, which exists by \cref{lemma:forestexist}. Let $H'$ be a digraph obtained by adding an arbitrary out-arc to every vertex of $H$ which does not yet have an out-arc in $H$ but does have at least one out-arc in $D_\theta$. Finally, let $F$ denote the underlying graph of $H'$, and note that $F$ is a pseudoforest (since it admits an orientation of maximum out-degree $1$). Note that the only vertices of $H'$ which do not have an out-arc are the vertices that form singleton terminal components of $D_\theta$. 

We now claim that $m_2(G-F)\le m$. 
Towards a contradiction, suppose that $m_2(G-F)>m$. Then there exists a set $U\subseteq V(G)$ such that $|U|\ge 3$ and
\begin{align}\label{ineq: basis}
    m(|U|-2)+1<\e_{G-F}(U).
\end{align}
We make the following observation.
\begin{claim}\label{clm: U is big}
    $|U|>2m-1.$
\end{claim}
\begin{proof}
    Suppose $|U|\leq 2m-1$. Together with \eqref{ineq: basis}, we get
    $$
        \binom{|U|}{2}\geq \e_{G-F}(U)>m(|U|-2)+1\geq\frac{|U|+1}2 (|U|-2)+1.
    $$
    But $|U|(|U|-1)=(|U|+1)(|U|-2)+2$, a contradiction.
\end{proof}
Let $d$ denote the number of terminal components of $D_\theta$ that intersect $U$ in exactly one vertex.

\begin{claim}\label{clm: d is 1}
    $d\leq 1$.
\end{claim}
\begin{proof}
    As $G_\theta^{\mathrm{frac}}$ forms a forest by \cref{lem: forest}, there are at most $|U|-1$ fractional edges in $G[U]$. For every vertex $u \in U$ such that there exists a terminal component $W$ of $D_\theta$ satisfying $W \cap U=\{u\}$, there are no arcs in $D_\theta$ from $u$ to $U$. Note that there are exactly $d$ such vertices in $U$, by definition of $d$. Furthermore, every other vertex in $U$ (for which such a terminal component does not exist) has out-degree at most $m-1$ in $D_\theta-H'$, since each such vertex has at most $\lfloor m\rfloor \le m$ out-arcs in $D_\theta$ by definition of $\theta$, and exactly one out-arc is included in $H'$. Bounding the number of integral edges in $(G-F)[U]$ by the sum over the out-degrees of vertices in $(D_\theta-H')[U]$, we get
    $$
    \e_{G-F}(U)\leq (|U|-1)+(m-1)(|U|-d).
    $$
    Comparing this with \eqref{ineq: basis}, we obtain
    $$-2m+2<-(m-1)d,$$
    implying that $d<2$.
\end{proof}
Since for every edge $uv\in E(G)$ it holds that $\theta(u,v)+\theta(v,u)=1$, we find that
\begin{align}\label{ineq: flowsum}
    \e_{G-F}(U)= \sum_{uv\in E((G-F)[U])}{(\theta(u,v)+\theta(v,u))}=\sum_{u\in U}\sum_{v\in N_{G-F}(u)\cap U}\theta(u,v).
\end{align}
Let $X$ be the set of vertices that form singleton terminal components of $D_{\theta}$. Then, by definition of an $m$-allocation and since every vertex outside of $X$ has out-degree $1$ in $H'$, we find that for every $u\in V(G)\setminus X$, we have $\sum_{v\in N_{G-F}(u)\cap U}\theta(u,v) \leq m-1$, whereas for every $u\in X$, we have $\sum_{v\in N_{G-F}(u)\cap U}\theta(u,v)\leq m$.

For now, suppose that for all $u\in U$, we have $\sum_{v\in N_{G-F}(u)\cap U}\theta(u,v)\leq m-1$. Then \eqref{ineq: flowsum} together with \eqref{ineq: basis} gives
$$
    m(|U|-2)+1<(m-1)|U|.
$$
Simplifying, we have $|U|<2m-1$, contradicting \cref{clm: U is big}.
Therefore, there exists at least one vertex $u\in U$ with $\sum_{v\in N_{G-F}(u)\cap U}\theta(u,v)> m-1$. Fix such a vertex and denote it by $u^*$. Observe that $u^*\in U\cap X$, which implies by \cref{clm: d is 1} that $d=1$. In particular, we find that $u^*$ is the unique vertex in $U$ satisfying $\sum_{v\in N_{G-F}(u)\cap U}\theta(u,v)> m-1$. Note that $u^*$ does not have any out-arcs in $D_\theta$. We conclude that $u^*$ is incident to at least $\lfloor m\rfloor$ edges in $G_\theta^{\mathrm{frac}}[U]$. Picking $k\in\mathbb{N}$ and $0\leq\varepsilon<1$ such that $m=k+\varepsilon$, we obtain that $u^*$ is incident to at least $k$ edges in $G_\theta^{\mathrm{frac}}[U]$.
\begin{claim}\label{clm: size of U}
    $2m-1<|U|<2m$.
\end{claim}
\begin{proof}
    The lower bound follows from \cref{clm: U is big}.
    Again combining \eqref{ineq: basis} and \eqref{ineq: flowsum}, we get
    $$
        m(|U|-2)+1<\e_{G-F}(U)\leq(m-1)|U|+1,
    $$
    where we use the fact that $u^*$ is the only vertex in $U$ satisfying $\sum_{v\in N_{G-F}(u)\cap U}\theta(u,v)> m-1$.
    The upper bound follows directly.
\end{proof}
\begin{claim}
    $\varepsilon<1/2$.
\end{claim}
\begin{proof}
    Suppose towards a contradiction that $\varepsilon\geq 1/2$. Since by \cref{clm: size of U}, $2m-1<|U|<2m$, we get that $\varepsilon>1/2$ and $|U| = 2k+1$. It follows that 
    \begin{equation*}
         \left(k+\frac 12\right)(2k-1)+1<(k+\varepsilon)(2k-1)+1=m(|U|-2)+1<\e_{G-F}(U),
    \end{equation*}
    where we make use of \eqref{ineq: basis} to get the last inequality. On the other hand, every edge in $(G-F)[U]$ corresponds to either an edge of $G_\theta^{\mathrm{frac}}$ or an arc of $D_\theta-H'$. By \cref{lem: forest}, $G_\theta^{\mathrm{frac}}$ is a forest and thus, contributes at most $|U|-1$ edges to $(G-F)[U]$. Since every vertex in $U\setminus \{u^\ast\}$ has at most $k-1$ integral out-arcs in $D_\theta-H'$, and $u^*$ has no out-arcs at all in $D_\theta-H'$, we get that at most $(|U|-1)\cdot (k-1)$ edges of $(G-F)[U]$ correspond to an edge of $D_\theta-H'$. Using these two observations, we get
    \begin{equation*}
    \e_{G-F}(U)\le (|U|-1)+(|U|-1)\cdot (k-1)=2k+2k\cdot (k-1)=2k^2.
    \end{equation*}
    Combining the above lower and bound bound on $e_{G-F}(U)$, we get 
    \begin{equation*}
    \left(k+\frac 12\right)(2k-1)+1<\e_{G-F}(U)\le 2k^2,
    \end{equation*}
    which is easily seen to be a contradiction by subtracting $2k^2$ from both sides.
\end{proof}
So suppose $\varepsilon<1/2$. By \cref{clm: size of U}, we get that $|U|=2k$. Recall that $u^*$ is incident to at least $k$ fractional edges in $G_\theta^{\mathrm{frac}}[U]$. By \cref{lemma: no_edge} and the fact that $d\leq 1$, there do not exist $v,w \in U$ such that both $u^*v$ and $u^*w$ are fractional edges, and such that $(v,w) \in A(D_\theta)\setminus A(H')$. Similarly, by \cref{lem: forest}, there are no fractional edges between fractional neighbours of $u^*$. Thus, the fractional neighbourhood of $u^*$ in $U$ induces an independent set in $G-F$, and therefore
$$
\e_{G-F}(U)\leq \binom{2k}{2}-\binom{k}{2}.
$$
Together with \eqref{ineq: basis}, we have
$$
k(2k-2)+1\leq m(2k-2)+1<k(2k-1)-\frac{k(k-1)}2.
$$
Simplifying, we get
$$
2<k(3-k).
$$
This is a contradiction for every 
$k\in \mathbb{N}$, concluding the proof.
\end{proof}

\section{Main technical lemma for Theorem \ref{thm:43}}\label{sec:technical lemma}
We now begin the proof of \cref{thm:43}, which is the content of this section and \cref{sec:43}.
In this section, we prove the following technical lemma, which allows us to argue that a minimal violating set $U$ must have certain structural properties with respect to the terminal components of $D_\theta$, and moreover allows us to dispense of most values of $m$.
\begin{lemma}\label{lemma: main techincal}
    Let $G$ be a graph with an optimal $m$-allocation $\theta$ for some $m>3/2$. Let $H$ be any spine of $D_\theta$ and $F$ the underlying undirected graph of $H$. Let $U\subseteq V(G)$ be an inclusion-wise minimal set of vertices such that $|U|\ge 2$ and $\e_{G-F}(U)\geq m(|U|-4/3)$ (assuming such a set exists). Then the following hold. 
    \begin{itemize}
        \item Either $m \leq \frac95$ or $m=\frac94$.
        \item $U$ intersects exactly one terminal component of $D_\theta$, and it intersects this component in at least two vertices.
        \item $|U| \leq \frac 43 m + 1$.
    \end{itemize}
\end{lemma}
\begin{proof}
By assumption, we have that $|U|\geq 2$ and
\begin{align}\label{tech ineq: basis}
    \e_{G-F}(U)\geq m\left(|U|-\frac 43\right).
\end{align}
Note that the above inequality together with $m>3/2$ implies that $|U|>2$.
Let us consider the collection $\mathcal{C}$ consisting of all strong components of $D_\theta$ that intersect $U$ in at least one vertex. Let us further split $\mathcal{C}$ into three parts, namely $\mathcal{C}_1$, consisting of all non-terminal components in $\mathcal{C}$, $\mathcal{C}_2$, consisting of all terminal components in $\mathcal{C}$ that contain at least two vertices of $U$, and $\mathcal{C}_3$, consisting of all terminal components in $\mathcal{C}_3$ that intersect $U$ in precisely one vertex. Clearly, the intersections $(X \cap U)_{X\in \mathcal{C}}$ form a partition of $U$. In the following, let us denote $U'\coloneqq \bigcup_{X\in \mathcal{C}_1}{(X \cap U)}$ and let $c\coloneqq |\mathcal{C}_2|,\ d\coloneqq |\mathcal{C}_3|$.

\subsection{Upper bounds on \texorpdfstring{$\e_{G-F}(U)$}{e\_\{G-F\}(U)}}
An essential strategy towards proving \cref{lemma: main techincal} is to compare various upper bounds on $\e_{G-F}(U)$ to find a contradiction with \eqref{tech ineq: basis}. In this subsection we give several upper bounds on $\e_{G-F}(U)$ which will be useful later.

Recall that the edges of $G$ are partitioned into $E_{\text{frac}}(G,\theta)$ and $E_{\text{int}}(G,\theta)$. Therefore,
\begin{align}\label{tech ineq: basic frac + integ}
    \e_{G-F}(U) = |E(G_\theta^\text{frac}[U])|+|A((D_\theta-H)[U])|.
\end{align}
Since $G^\text{frac}_\theta$ is a forest by \cref{lem: forest}, there are at most $|U|-1$ fractional edges. Additionally, we know that every vertex has at most $\lfloor m\rfloor$ out-arcs in $D_\theta$ by definition of $\theta$, at most $\lfloor m\rfloor-1$ out-arcs in $(D_\theta-H)[U]$ if it is not the root of a terminal component in $\mathcal{C}_2$ by \cref{def:spine}, and no out-arcs in $D_\theta[U]$ if it is a vertex in $\mathcal{C}_3$. Thus, we get
\begin{align}\label{tech ineq: frac + integ}
    \e_{G-F}(U)\leq |U|-1 + (\lfloor m\rfloor-1)(|U|-d)+c.
\end{align}
Yet another way of counting $\e_{G-F}(U)$, similarly as in \eqref{ineq: flowsum} in the proof of \cref{thm:pseudoforest}, is by summing up $\theta(u,v)+\theta(v,u)=1$ for each $uv\in E((G-F)[U])$. That is,
\begin{align}\label{tech ineq: theta sum}
    \e_{G-F}(U)=\sum_{(u,v) \in U^2, uv\in E(G-F)}{\theta(u,v)}=\sum_{u \in U}\sum_{v \in N_{G-F}(u)\cap U}{\theta(u,v)}.
\end{align}

\subsection{Upper and lower bounds on \texorpdfstring{$|U|$}{|U|}}
Comparing the inequalities of the previous subsection, we can obtain bounds on the size of $U$. Our next few claims provide such bounds, in terms of $c,d,m$, and the size of $U'$.
\begin{claim}\label{claim: u>2m-1}
    $|U|>2m-1.$
\end{claim}
\begin{proof}
    Suppose $|U|\leq 2m-1$. Together with \eqref{tech ineq: basis}, we get
    $$
        \binom{|U|}{2}\geq \e_{G-F}(U) \geq m\left(|U|-\frac 43\right)\geq\frac{|U|+1} 2 \left(|U|-\frac 43\right),
    $$
    which simplifies to $|U| \leq 2$, a contradiction.
\end{proof}

The following claim uses the above inequality to get an upper bound on the size of $U$.
\begin{claim}\label{claim: uupper}
    $|U|\le \frac{4}{3}m+c+d$. 
\end{claim}
\begin{proof}
Consider a vertex $u \in U$. Since $\theta$ is an $m$-allocation, we have $\sum_{v \in N_G(u)}{\theta(u,v)} \le m$. Therefore, if $u$ has out-degree one in $H$, then $$\sum_{v \in N_{G-F}(u)}{\theta(u,v)}\le m-1.$$ By \cref{def:spine}, every vertex in $U'$ has out-degree $1$ in $H$, and at most one vertex in each of the sets $(U \cap X)_{X\in \mathcal{C}_2\cup \mathcal{C}_3}$ can have out-degree $0$ in $H$, so in total there can be at most $c+d$ such vertices in $U$. Together with \eqref{tech ineq: theta sum}, this implies
$$\e_{G-F}(U)\le (m-1)(|U|-c-d)+m(c+d)=(m-1)|U|+c+d.$$
Combining with \eqref{tech ineq: basis}, we find
$$m\left(|U|-\frac{4}{3}\right)\le (m-1)|U|+c+d,$$
and rearranging yields $|U|\le\frac{4}{3}m+c+d$, as claimed.
\end{proof}
Observe that the previous claim implies that $|U|\leq \frac{4}{3}m+1$ if $(c,d)=(1,0)$, which was the third claimed result in the statement of \cref{lemma: main techincal}. Therefore, to complete the proof, it remains to prove that $(c,d)=(1,0)$ and either $m<9/5$ or $m=9/4$.

\begin{claim}\label{claim:ulower} If $(c,d)\neq(1,0)$, then $|U|\ge \frac{4}{3}mc+md+|U'|-\frac{4}{3}m+1$. If additionally $c>0$, then $|U|>\frac{4}{3}mc+md+|U'|-\frac{4}{3}m+1$. \end{claim}
\begin{proof}
Suppose that $(c,d)\neq (1,0)$, and let us prove the stated inequality. By the minimality of $U$, we have that $\e_{G-F}(U)\ge m(|U|-\frac{4}{3})$ but $\e_{G-F}(Y)< m(|Y|-\frac{4}{3})$ for every $Y\subsetneq U$ with $|Y|\ge 2$. Thus, we have $\e_{G-F}(X\cap U)< m(|X\cap U|-\frac{4}{3})$ for every $X \in \mathcal{C}_2$ since no $X\in \mathcal{C}_2$ fully contains $U$ as otherwise $(c,d)=(1,0)$. We furthermore trivially have $\e_{G-F}(X \cap U)=0$ for every $X\in \mathcal{C}_3$. 

Our goal in the following will be to give an upper bound on the number of all edges in $\e_{G-F}(U)$ based on \eqref{tech ineq: basic frac + integ}. As discussed earlier, $G_\theta^\text{frac}[U]$ is a forest by \cref{lem: forest} and hence contains at most $|U|-1$ edges. 

To bound the number of arcs in $(D_\theta-H)[U]$, we split the arc-set of this digraph into those arcs that start in the set $U'$ and those arcs that start in $U \setminus U'$. Since $H$ is a spine, every vertex in a non-terminal component has out-degree $1$ in $H$ and out-degree at most $m$ in $D_\theta$ by definition of $\theta$. Thus, we obtain that the number of arcs in $(D_\theta-H)[U]$ starting in $U'$ is upper-bounded by $(m-1)|U'|$. Secondly, since by definition all components in $\mathcal{C}_2$ and $\mathcal{C}_3$ are terminal and thus have no arcs in $D_\theta$ leaving them, the number of arcs in $(D_\theta-H)[U]$ starting in $U\setminus U'$ is at most $\sum_{X\in \mathcal{C}_2\cup \mathcal{C}_3}{\e_{G-F}(X\cap U)}$. Altogether, we obtain the following upper bound on the number of edges in $(G-F)[U]$:
$$\e_{G-F}(U)\le (|U|-1)+(m-1)|U'|+\sum_{X\in \mathcal{C}_2\cup \mathcal{C}_3}{\e_{G-F}(X\cap U)}.$$ Plugging in the upper bounds on $\e_{G-F}(X\cap U)$ for $X\in \mathcal{C}_2\cup \mathcal{C}_3$ mentioned above and comparing with \eqref{tech ineq: basis}, it follows that
\begin{align*}
    m\left(|U|-\frac{4}{3}\right)&\le (|U|-1)+(m-1)|U'|+\sum_{X \in \mathcal{C}_2}{m\left(|X\cap U|-\frac{4}{3}\right)}\\
    &=(|U|-1)+(m-1)|U'|+m(|U|-|U'|-d)-\frac{4}{3}mc\\
    &=(m+1)|U|-|U'|-\frac{4}{3}mc-md-1.
\end{align*}
Note that the inequality above is strict if $c>0$. Rearranging now yields that $|U|\ge \frac{4}{3}mc+md+|U'|-\frac{4}{3}m+1$ with strict inequality if $c>0$, as desired.
\end{proof}
\subsection{Consequences}
Combining \cref{claim:ulower,claim: uupper} now directly yields the following.

\begin{claim}\label{claim:cdinequality}
If $(c,d)\neq (1,0)$ then $(\frac{4}{3}m-1)c+(m-1)d+|U'|\le \frac{8}{3}m-1$, with a strict inequality if $c>0$. 
\end{claim}
\begin{proof}
Combining the lower bound on $|U|$ from \cref{claim:ulower} and the upper bound on $|U|$ from \cref{claim: uupper} yields that
$$\frac{4}{3}mc+md+|U'|-\frac{4}{3}m+1 \le |U|\le \frac{4}{3}m+c+d.$$ Rearranging now implies
$$\left(\frac{4}{3}m-1\right)c+(m-1)d+|U'|\le \frac{8}{3}m-1.$$
Furthermore, if $c>0$ then the inequality in \cref{claim:ulower} is strict, yielding a strict inequality.
\end{proof}

Using the previous arguments, we are able to obtain a case distinction into a finite number of options for the pair $(c,d)$. The rest of the proof will involve dealing with each case in turn.

\begin{claim} \label{claim:CDBOUND}
We have that $c\le 2$. Furthermore, the following hold:
\begin{itemize}
    \item If $c=0$, then $d \le 1$.  
    \item If $c=1$, then $d\le 3$ and $d\le 2$ provided $m\ge 9/5$. 
    \item If $c=2$, then $d\le 1$ and $d=0$ provided $m\ge 2$. 
\end{itemize}
\end{claim}
\begin{proof}
The claim trivially holds if $(c,d)=(1,0)$, so assume in the following that $(c,d)\neq (1,0)$.

Suppose first towards a contradiction that $c=0$ and $d \ge 2$. By \cref{claim:ulower} we have $|U| \ge md+|U'|-\frac{4}{3}m+1$. Since we also have $|U|=|U'|+d$, it follows that $md+|U'|-\frac{4}{3}m+1\le |U'|+d$, and thus $(m-1)d\le \frac{4}{3}m-1$. This is a contradiction, since $(m-1)d\ge 2(m-1)>\frac{4}{3}m-1$ for $m >\frac{3}{2}$. 

Next suppose that $c>0$. This implies, by \cref{claim:cdinequality} that $$\left(\frac{4}{3}m-1\right)c+(m-1)d+|U'|< \frac{8}{3}m-1,$$ and thus 
\begin{equation}\label{eq:d ub}
    d<\frac{\frac{4}{3}m(2-c)+c-1}{m-1}. 
\end{equation}

Suppose first that $c=1$. Then \eqref{eq:d ub} implies $d<\frac{\frac{4}{3}m}{m-1}<\frac{2}{\frac{3}{2}-1}=4$, so $d\le 3$, as claimed. If $m>\frac{9}{5}$, we obtain the stronger bound $d<\frac{12/5}{4/5}=3$, so $d\le 2$. 

Similarly, if $c=2$, then $d<\frac{1}{m-1}<\frac{1}{\frac{3}{2}-1}=2$, and so $d\le 1$. If $m\ge 2$, we obtain the stronger bound $d<1$ and thus $d=0$. 

If $c\ge 3$, we obtain $d<\frac{2-\frac{4}{3}m}{m-1}<0$, a contradiction, and thus we must have $c\le 2$. 
\end{proof}
Let us now assume towards a contradiction that one of $(c,d)\neq (1,0)$, $9/5<m<9/4$ or $m>9/4$ holds. Based on this and the previous claim, the proof now splits into the following cases, which we list in the order that we resolve them.
\begin{enumerate}[label=(\alph*)]
\item\label{(0,0)} $(c,d) = (0,0)$,
\item\label{(1,3)} $(c,d) = (1,3)$ and $m<9/5$,
\item\label{(1,0)_large} $(c,d) = (1,0)$ and $m>9/4$,
\item\label{(1,0)_small} $(c,d) = (1,0)$ and $9/5<m<9/4$,
\item\label{(2,1)} $(c,d) = (2,1)$ and $m<2$,
\item\label{(2,0)} $(c,d) = (2,0)$,
\item\label{(1,2)} $(c,d) = (1,2)$,
\item\label{(0,1)} $(c,d) = (0,1)$,
\item\label{(1,1)} $(c,d) = (1,1)$.
\end{enumerate}
\subsection{Easy arguments based on set sizes}
We now begin the case distinction to handle all the cases delineated in the above list. Many of the cases can be dealt with very directly by simply comparing the upper and lower bounds we have on the sizes of various sets. As all these arguments are simple and similar, we collect them all in this subsection.
\paragraph{\textbf{Case \ref{(0,0)}}}
    If $(c,d)=(0,0)$, then from \cref{claim: u>2m-1,claim: uupper} we obtain that
    $$2m-1<|U|\le \frac{4}{3}m+0+0,$$ which simplifies to $m<\frac{3}{2}$, a contradiction.

\paragraph{\textbf{Case \ref{(1,3)}}}
    Suppose $(c,d)=(1,3)$ and $3/2<m<9/5$. Then $|U|\geq 2c+d= 5$. Since there are no arcs in $D_\theta$ between different terminal components and $m<2$, the only vertex in $U$ with an out-arc in $D_\theta-H$ is the root of the component in $\mathcal{C}_2$ if it is in $U$. Otherwise there are no integral edges at all. Using that $G_\theta^{\mathrm{frac}}$ is a forest by \cref{lem: forest}, it follows from \eqref{tech ineq: basis} and \eqref{tech ineq: basic frac + integ} that
    $$
    \frac{3|U|}2-2<m\left(|U|-\frac 43\right)\leq \e_{G-F}(U)\leq (|U|-1)+1=|U|,
    $$
    where the first inequality uses that $m>3/2$. Rearranging yields that $|U|<4$, a contradiction.

\paragraph{\textbf{Case \ref{(1,0)_large}}}
    Suppose $(c,d)=(1,0)$ and $m>\frac{9}{4}$. By \cref{claim: uupper} we then have $|U|\le \frac{4}{3}m+1$. Using \eqref{tech ineq: basis}, it follows that
    $$m\left(|U|-\frac{4}{3}\right)\le \e_{G-F}(U)\le \binom{|U|}{2}=\frac{|U|-1}{2}|U|\le \frac{2}{3}m|U|.$$ Rearranging now yields that $|U|\le 4$. If $|U|=4$, then from the above we obtain $m(4-\frac{4}{3})\le 6$ and thus $m\le \frac{9}{4}$, a contradiction. Similarly, if $|U|=3$, then $m(3-\frac{4}{3})\le 3$, contradicting $m>\frac{9}{4}$.

\paragraph{\textbf{Case \ref{(1,0)_small}}}
    Suppose that $(c,d)=(1,0)$ and $9/5<m<9/4$. By \cref{claim: uupper}, we have $|U|\leq \frac 43 m+1<4$, so $|U|\leq 3$ and thus $|U|=3$. Using \eqref{tech ineq: basis}, we find $$\e_{G-F}(U)\geq m\left(|U|-\frac 43\right) = \frac{5m}3 > 3,$$ 
    which gives a contradiction since $\e_{G-F}(U)$ can contain at most $3$ edges.
\paragraph{\textbf{Case \ref{(2,1)}}}
    Suppose $(c,d)=(2,1)$ and $3/2<m<2$. Note that $|U|\geq 2c+d= 5$. By \cref{claim: uupper}, we get
    $$
    |U|\leq \frac{4}{3}m+3<6.
    $$
    It follows that $|U| = 5$. Note that this implies that $U$ shares exactly two vertices with two terminal components of $D_\theta$ and exactly one vertex with a third terminal component of $D_\theta$. Observe that there can be at most two arcs in $D_\theta[U]$ since there are no arcs between different terminal components. By \eqref{tech ineq: basic frac + integ} together with the fact that $G_\theta^\text{frac}$ is a forest by \cref{lem: forest}, we get
    $$
    \e_{G-F}(U)\leq 4 + 2 = 6.
    $$
    By \eqref{tech ineq: basis}, we get
    $$
    \frac{11}2<m\left(5-\frac43\right)\leq \e_{G-F}(U).
    $$
    It follows that $(G-F)[U]$ contains exactly two integral edges and four fractional edges, that is, the fractional edges are a tree. Let $(u,v)\in A(D_\theta[U])$ be an arc corresponding to one of those two integral edges. Since the fractional edges form a tree on $U$, there exists a path $P$ in $G^\text{frac}_{\theta}[U]$ from $u$ to $v$. Note that each of these fractional edges connects two different terminal components of $D_\theta$, since all the edges in $G[U]$ within a terminal component are contained in $D_\theta$. The existence of such a path contradicts the conclusion of \cref{lemma: no triangle between terminal components}.

\subsection{Trickier arguments involving more careful counting}
In this section we resolve the remaining cases, which turn out to be more involved. Let us start with the following claim.
\begin{claim}\label{claim:Uprimeempty}
If $(c,d)= (1,2)$ or $(c,d)= (2,0)$, then $U'=\emptyset$.
\end{claim}
\begin{proof}
    By \cref{claim:cdinequality} we have $|U'|<\frac{8}{3}m-1-(\frac{4}{3}m-1)c-(m-1)d$. For $(c,d)= (1,2)$ and $(c,d)=(2,0)$ we have that $|U'|<2-\frac{2}{3}m< 2-\frac{2}{3}\cdot\frac{3}{2}=1$ and $|U'|<1$ respectively, so in either case we may conclude $U'=\emptyset$. 
\end{proof}

\paragraph{\textbf{Case \ref{(2,0)}}}
Suppose that $(c,d)=(2,0)$. By \cref{claim:Uprimeempty} we then have $U'=\emptyset$. 
Let $C_1, C_2$ denote the two terminal components that intersect $U$, and let $U_1=U\cap C_1$ and $U_2=U\cap C_2$. Note that, by definition of $c$, we have $|U_1|, |U_2|\geq 2$ and since $d=0$ and $U'=\emptyset$ we have $U_1 \cup U_2=U$.
If we denote by $\e_{G-F}(U_1,U_2)$ the number of edges with one endpoint in each $U_i$, we have that
$$\e_{G-F}(U) = \e_{G-F}(U_1) + \e_{G-F}(U_2) + \e_{G-F}(U_1,U_2)\le \binom{|U_1|}{2}+\binom{|U_2|}{2}+(|U|-1),$$
where we used that all the edges in $G-F$ going between $U_1$ and $U_2$ connect the two terminal components $C_1$ and $C_2$ and are thus fractional, and that by \cref{lem: forest} the fractional edges form a forest.

Furthermore, we claim that  the three estimates $\e_{G-F}(U_1)\leq\binom{|U_1|}{2}$, $\e_{G-F}(U_2)\leq\binom{|U_2|}{2}$ and $\e_{G-F}(U_1,U_2)\leq|U|-1$ cannot all be equalities simultaneously. Indeed, if $\e_{G-F}(U_1, U_2)=|U|-1$ then the edges going across $U_1$ and $U_2$ form a tree on $|U|\geq 4$ vertices, meaning there exists a vertex in either $U_1$ or $U_2$ incident to at least two such edges. But if both $U_1$ and $U_2$ are cliques, this would imply that
there exists a triangle in $G-F$ formed by two fractional edges that cross between $U_1$ and $U_2$ and one integral edge inside $U_1$ or $U_2$, a contradiction to \cref{lemma: no triangle between terminal components}. We therefore obtain 
$$\e_{G-F}(U) \le \binom{|U_1|}{2}+\binom{|U_2|}{2}+(|U|-2).$$
By convexity of $\binom x2$ we now get
\begin{align*}
&\e_{G-F}(U)\leq \binom 22 + \binom{|U|-2} {2} + |U| - 2\\
& =  \frac{(|U|-2)(|U|-3)}2 + |U|-1 \\
& =  \frac{(|U|-2)(|U|-1)}2 + 1.
\end{align*}
By \eqref{tech ineq: basis}, we have $\e_{G-F}(U)\ge m(|U|-\frac{4}{3})=m(|U|-2)+\frac{2}{3}m>m(|U|-2)+1$. The above then implies that $(|U|-1)/2>m$, or, equivalently, $|U|>2m+1$. This contradicts \cref{claim: uupper}, which guarantees that $|U|\le \frac{4}{3}m+2=2m+1+(1-\frac{2}{3}m)<2m+1$, using that $m>\frac{3}{2}$.

\paragraph{\textbf{Case \ref{(1,2)}}}
    Suppose $(c,d)=(1,2)$. Let $k\in \mathbb{N}$ and $0<\varepsilon\leq1$ be such that $m=k+\varepsilon$. 
    By \cref{claim:ulower}, we have that $|U|>2m+1$. But we also have $|U|\leq \frac 43 m+3$ by \cref{claim: uupper}, which implies that $|U|< 2m+2$ since  $m>3/2$.
    We thus get that \begin{align}\label{ineq: U size (1,2)}
        2m+1<|U|<2m+2.
    \end{align}
    This is impossible if $m$ is an integer, so we may assume that $0<\varepsilon<1$, and thus that $k=\lfloor m \rfloor$.
    By \cref{claim:Uprimeempty}, we know that $U'$ is empty. Let $u_1,u_2$ be the two vertices from components in $\mathcal{C}_3$ and let $X$ be the remaining vertices in $U$, which are vertices of the component in $\mathcal{C}_2$. Note that $u_1,u_2$ are incident to no arcs in $D_\theta[U]$ since $U'$ is empty and no arcs of $D_\theta$ connect distinct terminal components. 
    
    Suppose first that $1/2<\varepsilon< 1$. By \eqref{ineq: U size (1,2)}, we have that $|U|=2k+3$. Using \eqref{tech ineq: basis} and \eqref{tech ineq: frac + integ}, we get
    $$
    \left(k+\frac 12\right)\left(|U|-\frac 43\right)<m\left(|U|-\frac 43\right)\leq (|U|-1) + (k-1)(|U|-2)+1.
    $$
    Simplifying, we arrive at
    $$
    \frac{|U|-4/3}2+\frac{2k}3< 2.
    $$
    Using now that $|U|=2k+3$, we get
    $5k/3+5/6< 2$,
    which is equivalent to $k<7/10$. But $k\geq 1$, a contradiction.

    Suppose then that $0<\varepsilon\leq 1/2$. By \eqref{ineq: U size (1,2)}, we get $|U| = 2k+2$. If $u_1$ is incident to at most $k$ fractional edges in $U$, then $U \setminus \{u_1\}$ contradicts the minimality of $U$ with respect to $e_{G-F}\geq m(|U|-4/3)$ (and the same holds for $u_2$). So we have that $u_1,u_2$ are each incident to at least $k+1$ edges in $G_\theta^{\mathrm{frac}}[U]$. Recall that by \cref{lem: forest} $G_\theta^{\mathrm{frac}}$ is a forest. First, suppose that $u_1u_2\notin E(G_\theta^{\mathrm{frac}})$. Thus, they both have at least $k+1$ edges to $X$ in $G_\theta^{\mathrm{frac}}$. Using that $|U|-2 = |X| = 2k$, we conclude that $|N_{G_\theta^{\mathrm{frac}}}(u_1)\cap N_{G_\theta^{\mathrm{frac}}}(u_2)|\geq 2$, yielding a copy of $C_4$ in $G_\theta^{\mathrm{frac}}$ and contradicting that $G_\theta^{\mathrm{frac}}$ is a forest. So we may assume that $u_1u_2$ is an edge of $G_\theta^{\mathrm{frac}}$. Then $|N_{G_\theta^{\mathrm{frac}}}(u_1)\cap N_{G_\theta^{\mathrm{frac}}}(u_2)| = 0$, for otherwise $G_\theta^{\mathrm{frac}}$ would contain a triangle, again contradicting \cref{lem: forest}. Therefore, each of the $2k$ vertices of $X$ is a neighbour of either $u_1$ or $u_2$. Since by \cref{lemma: no triangle between terminal components} the sets $N_G(u_1)\cap X, N_G(u_2)\cap X$ must be independent (and of size $k$), we obtain that $\e_{G-F}(X)\leq k^2$. It follows that
    $\e_{G-F}(U)\leq k^2+2k+1$.
    Comparing this to \eqref{tech ineq: basis}, we get
    $$
    k\left(2k+\frac 23\right)\leq k^2+2k+1,
    $$
    which is not true for any $k \geq 2$.
    But since $k+\varepsilon = m>\frac 32 \geq 1+\varepsilon$, we have that $k\geq 2$, a contradiction.

\paragraph{\textbf{Case \ref{(0,1)}}}
Suppose $(c,d)=(0,1)$. By \cref{claim: uupper,claim: u>2m-1}, we get
$$
2m-1<|U|\leq \frac{4}{3}m+1<2m,
$$
where we use $m>3/2$.
Let us fix $k\in\mathbb{N}$ and $0<\varepsilon\leq 1$ such that $m=k+\varepsilon$. Note that in fact, the bounds $2m-1<|U|<2m$ are impossible to achieve if $m$ is an integer, and therefore, we may assume that $\varepsilon<1$ and thus that $k=\lfloor m \rfloor$. Let $u^*$ be the vertex in the terminal component of $D_\theta$ which is intersected by $U$.

Suppose first that $1/2\leq\varepsilon< 1$. Since $2m-1<|U|<2m$, we get that $\varepsilon>1/2$ and $|U| = 2k+1$. Using \eqref{tech ineq: basis} and \eqref{tech ineq: frac + integ}, we get
    $$
    \left(k+\frac 12\right)\left(2k-\frac 13\right)\leq m\left(2k-\frac 13\right)<2k+(k-1)2k=2k^2,
    $$
    which we simplify to
    $$
    \frac{2k-1/3}2\leq \frac k3.
    $$
    But this is a contradiction, since $(2k-1/3)/2=k-\frac{1}{6}\geq \frac{5}{6}k$.
    
    So we may assume that $0<\varepsilon<1/2$. Then $k\geq 2$ and we find that $|U|=2k$. Suppose that $\sum_{v\in N_{G-F}(u^*) \cap U}\theta(u^*,v)\leq m-1$. Using \eqref{tech ineq: basis} together with \eqref{tech ineq: theta sum}, we get
    $$
    m\left(2k-\frac 43\right)\leq (m-1)2k.
    $$
    Simplifying yields 
    $$
    2k\leq \frac{4}{3}m<\frac{4}{3}k+\frac{2}{3},
    $$
    which does not hold for any positive integer $k$.
    It follows that $\sum_{v\in N_{G-F}(u^*) \cap U}\theta(u^*,v)> m-1\geq k-1$. Note that $u^*$ does not have an out-arc in $D_\theta[U]$ since it is part of a terminal component of $D_\theta$ intersecting $U$ in a single vertex. By definition of $\theta$, it follows therefore that if $u^*$ has an out-arc in $D_\theta$ then $u^*$ has at most $m-1$ outgoing weight to $U$ in $\theta$, contradicting the computation above. Thus, $u^*$ does not have any out-arcs in $D_\theta$ and hence, forms a singleton terminal component. 
    Thus, we obtain that $u^*$ must be incident to at least $k$ fractional edges in $G_\theta^{\mathrm{frac}}[U]$. By \cref{lemma: no_edge}, there cannot exist any arc of $(D_\theta-H)[U]$  between two neighbours of $u^*$ in $G_\theta^{\mathrm{frac}}$, since this would imply that one of those two neighbours also forms a singleton terminal component contained in $U$, contradicting our assumption that $d=1$. Also, no two neighbours of $u^*$ in $G_\theta^{\mathrm{frac}}[U]$ can be joined by a fractional edge, since $G_\theta^{\mathrm{frac}}$ is a forest. Hence, the endpoints of those $k$ fractional edges incident to $u^*$ in $G_\theta^{\mathrm{frac}}[U]$ form an independent set in $(G-F)[U]$. We conclude that
$$
\e_{G-F}(U)\leq \binom{2k}{2}-\binom{k}{2}.
$$
Together with \eqref{tech ineq: basis}, we have
$$
k\left(2k-\frac 43\right)< m\left(2k-\frac 43\right)\leq k(2k-1)-\binom k2.
$$
Simplifying, we get
$$
\binom{k}{2}<\frac k3.
$$
This is a contradiction for every 
$k\geq 2$.

\paragraph{\textbf{Case \ref{(1,1)}}}
    Suppose $(c,d)=(1,1)$. By \cref{claim: u>2m-1,claim: uupper}, we get 
    \begin{align*}
        2m-1<|U|\leq \frac{4}{3}m+2.
    \end{align*}
    Using additionally that $m>3/2$, we get that $\frac{4}{3}m+2<2m+1$ and hence, 
    \begin{align}\label{ineq: U size (1,1)}
        2m-1<|U|< 2m+1.
    \end{align}
    For each $u\in U$, let $I(u)$ denote the out-degree of $u$ in $(D_\theta-H)[U]$. 
    Let $X$ be the intersection of $U$ with the component in $\mathcal{C}_2$. Let $u^*\in U$ be the vertex in the component in $\mathcal{C}_3$ and $r\in X$ chosen such that $I(r)$ is maximal. Note that $X$ has no out-arcs in $D_\theta[U]$, as it is the intersection of $U$ with a terminal component. Fix $k\in \mathbb{N}$ and $0\leq \varepsilon<1$ such that\footnote{Note that this is a different choice from the one made in cases \ref{(1,2)} and \ref{(0,1)}, as we now allow $\varepsilon=0$ and earlier allowed $\varepsilon=1$.} $m=k+\varepsilon$. Then,\
    $I(u^*)=0$ and $I(r)\leq \lfloor m\rfloor= k$. Furthermore, we have $I(u)\leq k-1$ for every $u\in U\setminus\{u^*,r\}$. Indeed, otherwise there would be at least two vertices in $U\setminus \{u^\ast\}$ that have out-degree $k=\lfloor m\rfloor$ in $(D_\theta-H)[U]$. However, this is impossible, since by \cref{def:spine} all but at most one vertex of $U\setminus \{u^\ast\}$ have out-degree $1$ in $H$. Let 
    $$
    I \coloneqq k-I(r)+\sum_{u\in U\setminus\{u^*,r\}} (k-1-I(u)).
    $$
    Intuitively, $I$ is the number of integral edges that are ``missing''.
    Now, using that the number of integral edges in $(G-F)[U]$ is equal to $\sum_{u \in U}I(u)$, we can rewrite \eqref{tech ineq: basic frac + integ} as
    \begin{align*}
        \e_{G-F}(U)&\leq|U|-1+\sum_{u\in U}I(u)\\&\leq|U|-1+k+\left(\sum_{u\in U\setminus\{u^*,r\}}(k-1)\right) - I\\
        &=|U|-1+(k-1)(|U|-1)+1 -I\\&= k(|U|-1)+1-I.
    \end{align*}    
    Together with \eqref{tech ineq: basis}, we get
    \begin{align*}
        m\left(|U|-\frac 43\right)\leq k(|U|-1)+1-I.
    \end{align*}
    Plugging in  $m=k+\varepsilon$, we simplify to
    \begin{align}\label{ineq: infringement upper bound}
        I\leq \frac k3-\varepsilon\left(|U|-\frac 43\right)+1.
    \end{align}
    Let $M$ denote the number of edges in the complement of $(G-F)[U]$, so that
    $$
    \e_{G-F}(U)= \binom{|U|}{2}-M.
    $$
    Combining this with \eqref{tech ineq: basis}, we get
    \begin{align}\label{ineq: missing edges}
        M\leq\binom{|U|}{2}-m\left(|U|-\frac 43\right).
    \end{align}
    We now split into subcases depending on whether $0\leq \varepsilon\leq1/2$ or $1/2<\varepsilon<1$.
    First, suppose that $1/2<\varepsilon<1$. By \eqref{ineq: U size (1,1)}, we get that $|U|\in\{2k+1,2k+2\}$. However, plugging $\varepsilon> 1/2$ and $|U| \in \{2k+1, 2k+2\}$ into \eqref{ineq: infringement upper bound}  (and using the fact that $I$ must be non-negative) implies that the only possibility is $I = 0$, $k=1$ and $|U| = 2k+1=3$. 
    This in turn implies that $U\setminus \{u^\ast\}=X$. By \cref{lemma: no triangle between terminal components}, $(G-F)[U]$ cannot be a triangle, since this would imply that $u^\ast$ has two adjacent neighbours in the terminal component containing $X$. Hence, we have $m(3-\frac{4}{3})\le \e_{G-F}(U)\le 2$, a contradiction.

    So we may now assume that $0\leq \varepsilon\leq1/2$. Since $m>3/2$, we get $k\geq 2$. By \eqref{ineq: U size (1,1)}, we get that $|U|\in \{2k,2k+1\}$. Suppose first that $|U|=2k+1$. Plugging this into \eqref{ineq: missing edges}, and using that $k\leq m$, we get
    $$
    M\leq k(2k+1)- k\left (2k-\frac 13\right)=\frac{4}{3}k.
    $$
    Next, let us look at \eqref{ineq: infringement upper bound}. We get $I\leq k/3 +1< k,$ where we used $k\geq 2$. Hence, $I\leq k-1$.
    It follows that 
    \begin{equation}\label{ineq: (1,1) e(X) lb}
        \e_{G-F}(X)\geq\sum_{u\in X}I(u)\geq (k-1)|X|+1-I\geq (k-1)(|X|-1)+1.
    \end{equation}
    It follows that $|X|\geq 2k-1$.
    Suppose now that $\sum_{v\in N_{G-F}(u^*)\cap U}\theta(u^*,v)\leq m-1$.
    Using \eqref{tech ineq: theta sum}, we get
    $$
    \e_{G-F}(U)\leq(m-1)|U|+1.
    $$
    Comparing this with \eqref{tech ineq: basis}, we get
    $$
        m\left(|U|-\frac 43\right)\leq(m-1)|U|+1.
    $$
    Simplifying gives $|U|\leq\frac{4}{3}m+1<2m$, contradicting $|U|=2k+1\geq2m$. So we may assume that 
    \begin{equation}\label{eq:big sum}
    \sum_{v\in N_{G-F}(u^*)\cap U}\theta(u^*,v)> m-1.
    \end{equation}
    Recall that $u^*$ does not have any out-arcs in $D_\theta[U]$. Additionally, the assumption \eqref{eq:big sum} implies that $u^*$ cannot have an out-arc in $D_\theta$ to a vertex outside of $U$. It follows that $u^*$ does not have any out-arcs in $D_\theta$. In particular, $u^*$ is a singleton terminal component of $D_\theta$. Additionally, $u^*$ is incident to at least $k$ edges in $G_\theta^{\mathrm{frac}}[U]$. Suppose that $u^*$ is incident to at least two edges in the complement of $(G-F)[U]$. The fact that $d=1$ implies by \cref{lem: forest,lemma: no_edge} that the fractional neighbourhood of $u^*$ in $U$ is an independent set, so we find that 
    $$
    M\geq\binom{k}{2}+2.
    $$
    It follows that
    $$
    \binom{k}{2}+2\leq\frac{4}{3}k,
    $$
    which holds for no $k\in\mathbb{N}$. So we may assume that $u^*$ has at most one non-neighbour in $(G-F)[U]$. Then, using that there are no integral edges between $u^*$ and $X$, we get that $u^*$ is connected by at least $|X|-1$ fractional edges to $X$. Since by \cref{lemma: no triangle between terminal components} the neighbours of $u^\ast$ in $X$ form an independent set in $G$, it follows that $G[X]$ contains at most $|X|-1$ edges. On the other hand, $(G-F)[X]$ must contain at least $(k-1)(|X|-1)+1$ integral edges by \eqref{ineq: (1,1) e(X) lb}. So we get
    $$
     (k-1)(|X|-1)+1\leq |X|-1,
    $$
    implying that $k\le 1$, a contradiction.

    We move on to the last case, that is $|U|=2k$. We again plug this into \eqref{ineq: missing edges} and \eqref{ineq: infringement upper bound} to get, using $k\geq 2$,
    $$ 
    M\leq k(2k-1)-k\left(2k-\frac43\right) = \frac k3<k-1,
    $$
    and
    $$
        I\leq \frac k3+1< k.
    $$
    As above, this implies that $X$ has at least $2k-1$ vertices. Since there are at most $k-2$ missing edges, we get that $u^*$ is adjacent to at least $k+1$ edges into $|X|$. Using \cref{lemma: no triangle between terminal components}, we get that there are at least $\binom{k+1}{2}$ missing edges and thus,
    $$ 
        \binom{k+1}{2}\leq M\leq k-2,
    $$
    a contradiction.
This concludes the proof of the last remaining case and hence, of \cref{lemma: main techincal}.
\end{proof}

\section{Proof of Theorem \ref{thm:43}}\label{sec:43}
Given \cref{lemma: main techincal}, we now proceed with the proof of \cref{thm:43}. We use \cref{lemma: main techincal} to deal with the majority of cases, whereas a more careful selection of the forest $F$ is needed to resolve the final cases.

\begin{proof}[Proof of \cref{thm:43}]
We begin by picking an optimal $m$-allocation $\theta$ of $G$, which we may do by \cref{thm:allocations}.
We now split the proof into three cases, depending on the value of $m$. 
\paragraph{\textbf{Case 1: ${9/5<m<9/4}$ or $m>9/4$}}
By \cref{lemma:forestexist}, $D_\theta$ contains a spine $H$ with underlying undirected graph $F$, which is a forest. It then follows by \cref{lemma: main techincal} that for every $U\subseteq V(G)$ with $|U|\geq 2$, it holds that $\e_{G-F}(U)<m(|U|-4/3)$ and the statement follows. 

\paragraph{\textbf{Case 2: $m=9/4$}}
Again, by \cref{lemma:forestexist}, there exists a spine of $D_\theta$. Suppose that $H$ is a spine of $D_\theta$ minimizing the number of copies of $K_4$ in $G-F$, where $F$ is the underlying undirected graph of $H$. We are done if there is no $U\subseteq V(G)$ with $|U|\geq 2$ and $\e_{G-F}(U)\geq m(|U|-4/3)$, so we may assume that $U$ is an inclusion-wise minimal such set. By \cref{lemma: main techincal}, we get that $U$ intersects exactly one terminal component $X$ of $D_\theta$ and it intersects $X$ in at least two vertices as well as that $|U|\leq \frac 43m+1=4$. 
We note that if $|U|=2$, then our assumption reads $\e_{G-F}(U) \geq m(|U|-4/3) = 2m/3 = 3/2$, which is impossible since $\e_{G-F}(U) \leq 1$ if $|U|=2$. Similarly, the case $|U|=3$ can be ruled out as our assumption reads
$\e_{G-F}(U)\geq m(3-4/3)=5m/3=15/4$, which is impossible as $15/4>3$. So we may conclude that $|U|=4$, and moreover we have that $\e_{G-F}(U) \geq m(4-4/3)=8m/3=6$, which implies that $U$ induces a $K_4$ in $G-F$.

Let $W$ be any set of vertices of $G$. As in \eqref{ineq: flowsum}, we have that
\begin{align}\label{eq: theta_sum}
\e_{G-F}(W)=\sum_{(u,v) \in W^2, uv\in E(G-F)}{\theta(u,v)}=\sum_{u \in W}\sum_{v \in N_{G-F}(u)\cap W}{\theta(u,v)}.
\end{align}
Let us apply the above equality to $U$. Let $r$ denote the unique vertex in $X$ that has out-degree $0$ in $H$  (i.e., the root of the in-tree that $H$ induces on $X$). Since $H$ contains an out-arc
for every vertex in $U\setminus \{r\}$, it follows that 
$\sum_{v \in N_{G-F}(r)\cap U}{\theta(r,v)}\leq m$ and that  $\sum_{v \in N_{G-F}(u)\cap U}{\theta(u,v)}\leq m-1$ for all $u\in U\backslash\{r\}$.
This implies that $r\in U$, as otherwise, $\e_{G-F}(U)\leq (m-1)|U| = 5$. For the same reason, $U$ must contain all out-neighbours of $r$ in $D_\theta$. Let $u\in U$ be an out-neighbour of $r$ and observe that $u\in X$ since $X$ is a terminal component. Thus, by \cref{def:spine}, there is a directed path from $u$ to $r$ in $H$. Therefore, $H+(r,u)$ contains a unique directed cycle, which necessarily includes the edge $(r,u)$. Let $v$ denote the successor of $u$ along this cycle. As $U$ induces a $K_4$ in $G-F$, and as $uv\notin E(G-F)$, we have that $v\notin U$. Moreover, $H+(r,u)-(u,v)$ is a spine of $D_\theta$ with root $u$ in $X$. As $H$ was chosen to minimize the number of copies of $K_4$ in $G-F$, it must hold that there exists some vertex set $\widetilde{U}$ that induces a $K_4$ in $G-(F+ru-uv)$ but not in $G-F$. The same argument as above implies that $u, v \in \widetilde{U}$ and $r\notin \widetilde{U}$. Moreover, we have $\e_{G-F}(\widetilde{U})\ge \e_{G-(F+ru-uv)}(\widetilde{U})-1 = 6-1=5$. Observe that $\widetilde{U}$ is also a minimal set satisfying $|\widetilde{U}|\geq 2$ and $\e_{G-(F+ru-uv)}(\widetilde{U})\geq m(|\widetilde{U}|-4/3)$. By \cref{lemma: main techincal}, it follows that $\widetilde{U}$ only intersects one terminal component of $D_\theta$; this terminal component is necessarily $X$, since $u\in X\cap \widetilde{U}$.

On the one hand, applying \eqref{eq: theta_sum} to $W=U\cup \widetilde{U}$, we get $\e_{G-F}(U\cup \widetilde{U})\leq (m-1)|U\cup \widetilde{U}|+1$, since every vertex in $U \cup \widetilde{U}$ distinct from $r$ has out-degree $1$ in $H$. On the other hand, $\e_{G-F}(U\cup \widetilde{U})\geq \e_{G-F}(U)+\e_{G-F}(\widetilde{U})-\e_{G-F}(U\cap \widetilde{U}) \geq 11 - \binom{|U\cap \widetilde{U}|}{2}.$ The size of $|U\cap \widetilde{U}|$ is at least $1$ as both sets contain $u$, and is at most $3$ since $r\in U \setminus \widetilde{U}$. Hence, $m=9/4$ must satisfy one of the inequalities
$$ 11 - \binom 12 \leq 7(m-1)+1,  $$
$$ 11 - \binom 22 \leq 6(m-1)+1,  $$
or
$$ 11 - \binom 32 \leq 5(m-1)+1. $$
But none of these inequalities hold, a contradiction which concludes the proof of this case.

\paragraph{\textbf{Case 3: $3/2<m\leq9/5$}}
We begin by noting that since $m<2$, every vertex of $G$ has out-degree at most one in $D_\theta$. The resolution of this case now rests on the following crucial claim.
\begin{claim}\label{clm: good edge} 
Each non-singleton terminal component of $D_\theta$ contains an arc $(u,v)$ such that $u$ and $v$ have no common neighbours in $G_\theta^{\mathrm{frac}}$. 
\end{claim}
Before we prove the above claim, let us first show how to use it to complete the argument. Suppose then that every non-singleton terminal component contains an arc $(u,v)\in A(D_\theta)$ such that $u$ and $v$ have no common neighbour in $G_\theta^\text{frac}$. Let $R$ be a set of vertices containing exactly one such vertex $u$ for each non-singleton terminal component, as well as all the vertices in singleton terminal components of $D_\theta$. By \cref{lemma:forestexist}, there exists a spine $H$ of $D_\theta$ with underlying forest $F$ such that every vertex in $V(D)\setminus R$ has out-degree exactly one in $H$ and therefore, out-degree zero in $D_\theta - H$. Suppose towards a contradiction that $m_{\frac{4}{3}}(G-F)\geq m$. Then, there exists a minimal $U\subseteq V(G)$ with $|U|\geq 2$ and $\e_{G-F}(U)\geq m(|U|-4/3)$. By \cref{lemma: main techincal}, it follows that $U$ intersects a unique terminal component $X$ of $D_\theta$ and that $U$ intersects $X$ in at least two vertices. Additionally, we have that $|U|\leq \frac 43 m + 1 < 4$. Observe that $|U|>2$ since $m(2-4/3)>1$, so we conclude that $|U|=3$. Additionally, $\e_{G-F}(U) \geq m(3-4/3)>2$, implying that $(G-F)[U]$ is a triangle. By \cref{lem: forest}, $G_\theta^\text{frac}$ is a forest and therefore $(D_\theta-H)[U]$ contains at least one arc. But recall that the only vertices with out-arcs in $(D_\theta-H)[U]$ are vertices of $R$, and that $U$ intersects only one terminal component $X$. Let $r$ be the root of $X$, so that $U\cap R = \{r\}$. Let $(r,v)$ be the out-arc of $r$ in $(D_\theta-H)[U]$, and note that by the argument above this is the only arc in $(D_\theta-H)[U]$. But by our choice of the roots, $r$ and $v$ do not share any neighbours in $G_\theta^\text{frac}$ and hence $(G-F)[U]$ can not be a triangle. This contradiction completes the proof, and all that remains is to prove \cref{clm: good edge}.
\begin{proof}[Proof of \cref{clm: good edge}]
    Let us start by proving the following lemma.
    \begin{lemma}\label{lemma:star} Let $T$ be a forest, $\ell\ge 2$ an integer, and let $v_1, v_2, \dots, v_\ell$ be pairwise distinct vertices in $T$ such that each of the vertex pairs $v_1v_2, v_2v_3, \dots, v_\ell v_1$ have a common neighbour in $T$.
    Then there exists a vertex $u$ that is adjacent to all vertices $v_1, \dots, v_\ell.$ 
    \end{lemma}
    \begin{proof}  We prove the lemma by induction on $\ell$. It trivially holds if $\ell=2$, so suppose $\ell\ge 3$ and that we have proved the lemma with value $\ell-1$. For each $i=1,\ldots,\ell$, let $w_i$ be a common neighbour of $v_i$ and $v_{i+1}$ (index addition modulo $\ell$). Let $W\coloneqq \{v_1,\ldots,v_\ell,w_1,\ldots,w_\ell\}$. Since $T[W]$ is a forest, it contains a vertex of degree at most $1$, which must be one of $v_i, i=1,\ldots,\ell$. Without loss of generality say that $v_\ell$ has degree $1$ in $T[W]$. This implies that $w_{\ell-1}=w_\ell$, as both are neighbours of $v_\ell$. Hence, $v_1,\ldots,v_{\ell-1}$ also has the property that any pair of two cyclically consecutive vertices have a common neighbour in $T$. By the inductive hypothesis, there is a common neighbour $w$ of $v_1,\ldots,v_{\ell-1}$. Since both $w$ and $w_{\ell-1}=w_\ell$ are common neighbours of $v_1$ and $v_{\ell-1}$, and since $T$ contains no cycles, it follows that $w=w_\ell=w_{\ell-1}$, and thus $w$ is a common neighbour of $v_1,\ldots,v_\ell$.
    \end{proof}
    Note that, as the maximum out-degree of $D_\theta$ is at most $1$, each non-singleton terminal component of $D_\theta$ is the vertex-set of a directed cycle. Suppose towards a contradiction that there exists a non-singleton terminal component $C$ such that for every arc $(u,v)$ on the directed cycle $D_\theta[C]$, the vertices $u$ and $v$ have a common neighbour in $G_\theta^{\mathrm{frac}}$. By \cref{lem: forest}, $G_\theta^{\mathrm{frac}}$ is a forest and therefore, by \cref{lemma:star}, there exists a common neighbour $u$ of all vertices of $C$ in $G_\theta^\text{frac}$. This implies that $u\notin C$. Moreover, by \cref{lemma: no triangle between terminal components}, $u$ cannot be contained in another terminal component. Recall that since $m<2$, $u$ has at most one out-arc in $D_\theta$, and if it had zero out-arcs it would be a singleton terminal component, which we just argued is not the case. Therefore, $u$ has exactly one out-arc in $D_\theta$, and thus $H$ contains a directed path from $u$ to a terminal component $C'$ of $D_\theta$.

    First, suppose that $C'\neq C$. Let $(v,w)\in A(D_\theta[C])$ be any arc and $\theta'$ the $m$-allocation obtained from $\theta$ by shifting along the triangle $u,v,w$. By \cref{obs:shift}, $D_{\theta'}$ is obtained from $D_\theta$ by removing the arc $(v,w)$ and adding one (or both) of the arcs $(v,u), (u,w)$. But as $u$ has an out-arc in $D_\theta$ which remains in $D_{\theta'}$, it cannot be that $(u,w)$ gets added, since this would contradict $m<2$ and the fact that $\theta'$ remains an $m$-allocation. Hence, $D_{\theta'}=D_\theta-(v,w)+(v,u)$. Note that by \cref{it:delete arc} and using that $v$ has no out-arcs in $D_\theta - (v,w)$, the terminal components of $D_\theta-(v,w)$ are the singleton $\{v\}$ and the remaining terminal components of $D_\theta$ besides $C$. But then, \cref{it:subgraph} implies that $D_{\theta'}=D_\theta-(v,w)+(v,u)$ has strictly fewer terminal components than $D_\theta$, since $(v,u)$ together with the path in $H$ from $u$ to $C'$ forms a directed path from $v$ to another terminal component of $D_\theta-(v,w)$. This contradicts the optimality of $\theta$.
    
    Therefore, we may assume that $C' = C$. Thus, $H$ contains a directed path from $u$ to $C$. Let $(v,w)\in A(D_\theta[C])$ be the directed edge in $C$ such that $w$ is the first vertex in $C$ reached on this path. Let $\theta'$ be the $m$-allocation obtained from $\theta$ by shifting along the triangle $u,v,w$. As before, we can use \cref{obs:shift} and the fact that $u$ has an out-arc in $D_\theta$ that remains in $D_{\theta'}$ (recall that $u$ is connected only by fractional edges to the vertices in $C$, and hence cannot have an out-arc to one of $v,w$) to conclude that $D_{\theta'}=D_{\theta}-(v,w)+(v,u)$. As above, by \cref{it:delete arc}, the terminal components of $D_\theta-(v,w)$ are the same as the ones of $D_\theta$, except that $C$ is replaced by the new singleton terminal component $\{v\}$. Then, it follows by \cref{it:singleton} and the optimality of $\theta$ that $D_{\theta'}$ has the same terminal components as $D_\theta$ besides $C$, which is replaced by a terminal component of $D_{\theta'}$ containing $\{u,v\}$. Let us denote the vertices of this component by $A$. Every vertex reachable from $v$ in $D_{\theta'}$ is contained in $A$. By our choice of $w$, we see that $w$ is reachable from $u$ and hence from $v$ (via the arc $(v,u)$) in $D_{\theta'}$. Therefore all vertices in $C$ are reachable, since $C-(v,w)$ is a path from $w$ to $v$ passing through all vertices of $C$. Therefore, $C\subsetneq A$, as $u \in A \setminus C$. Thus, there are strictly more vertices contained in terminal components of $D_{\theta'}$ than in terminal components of $D_\theta$. Since the number of terminal components of $D_\theta$ and $D_\theta'$ is identical and since the number of integral edges of $\theta$ and $\theta'$ is also identical, this is a contradiction to the optimality of $\theta$, completing the proof.
\end{proof}
As argued above, \cref{clm: good edge} finishes the last case, which concludes the proof of \cref{thm:43}.
\end{proof}

\section{Concluding remarks}\label{sec:conclusion}
The results of this paper confirm the Kohayakawa--Kreuter conjecture, \cref{conj:KK}, for all $r$-tuples of graphs. However, there remain a few natural directions for future work to explore.

The first question we wish to highlight is a conjecture of Kuperwasser, Samotij, and Wigderson \cite[Conjecture 1.5]{kuper}, which strengthens the conclusion of \cref{thm:pseudoforest} by demanding that $F$ is a forest, rather than a pseudoforest.
\begin{conjecture}[Kuperwasser--Samotij--Wigderson]\label{conj:KSW}
    Let $m>0$ be a real number. Every graph $G$ with $m(G) \leq m$ contains a forest $F \subseteq G$ such that $m_2(G-F) \leq m$.
\end{conjecture}
Note that this conjecture is in fact a common strengthening of both \cref{thm:pseudoforest} and \cref{thm:43}, since $m_{\frac 43}(G-F)<\max \{m,m_2(G-F)\}$ assuming $m> \frac 32$.
This is an interesting and natural graph decomposition in its own right, but would also have further implications for the study of Ramsey properties of random graphs. Namely, as proved in \cite[Theorem 1.10]{kuper}, \cref{conj:KSW} would imply a natural generalization of \cref{conj:KK} where one wishes to color $G_{n,p}$ and avoid not just a single graph $H_i$ in color $i$, but some finite family $\{H_i^{(1)},\dots,H_i^{(t)}\}$. It seems very possible that our techniques, potentially coupled with a more careful choice of $\theta$ and $F$, could resolve \cref{conj:KSW}. We note that \cref{conj:KSW} was proved in \cite{kuper} under the extra assumption that $m$ is an integer, using matroid-theoretic techniques; however, those techniques do not seem to work in the general case.

Another natural question that deserves more study concerns the Ramsey properties of random hypergraphs. In this setting, even the symmetric case (i.e.\ the generalization of \cref{thm:RR} to hypergraphs) is not well-understood. The analogous $1$-statement to that in \cref{thm:RR} is known to be true \cite{MR2760356,MR3548529}. However, it was shown in \cite{MR3725732} that in uniformity $4$ and above, the corresponding $0$-statement is not true in general; there exist hypergraphs whose symmetric Ramsey threshold is controlled by an intermediate ``reason'' between the ``local'' and ``global'' reasons discussed in the introduction. Despite this, Bowtell, Hancock, and Hyde \cite{bowtell} proved an analogue of \cref{thm:prob lem} in the hypergraph setting, so in some sense all that remains to understand is when the hypergraph analogue of \cref{thm:det lem} holds.

\section*{Acknowledgements}
We are indebted to Benjamin Moore for many insightful discussions on this topic, and for carefully reading an earlier draft of this paper. We would also like to thank the anonymous referees for many helpful comments and corrections.

\bibliographystyle{yuval}
\bibliography{kk}
\end{document}